\documentclass[10pt]{amsart}
\usepackage{amsfonts}
\usepackage{amssymb}
\usepackage{color}
\usepackage{amsmath}
\usepackage{bm}
\usepackage{graphicx}
\usepackage{float}
\usepackage[all,cmtip]{xy}
\newenvironment{wideitemize}
{ \begin{itemize}
    \setlength{\itemsep}{3pt}
    \setlength{\parskip}{0pt}
    \setlength{\parsep}{0pt}     }
{ \end{itemize}                  } 
\newcommand{\D}{{\mathbb{D}}}
\renewcommand{\S}{{\mathbb{S}}}
\newcommand{\R}{{\mathbb{R}}}

\newcommand{\Rot}{\operatorname{Rot}}

\newcommand{\Engel}{{\mathfrak{E}}}

\newcommand{\Flags}{{\mathfrak{F}}}

\newcommand{\SW}{{\mathcal{W}}}
\newcommand{\SD}{{\mathcal{D}}}
\newcommand{\SE}{{\mathcal{E}}}

\newcommand{\SH}{{\mathcal{H}}}

\newcommand{\FF}{{\mathfrak{F}}}

\newcommand{\SF}{{\mathcal{F}}}

\newcommand{\SL}{{\mathcal{L}}}

\newcommand{\NS}{{\mathbb{S}}}

\newcommand{\Op}{{\mathcal{O}p}}

\newcommand{\Id}{{\operatorname{Id}}}

\newcommand{\wt}{\widetilde}

\newtheorem{proposition}{Proposition}
\newtheorem{theorem}[proposition]{Theorem}
\newtheorem{definition}[proposition]{Definition}
\newtheorem{lemma}[proposition]{Lemma}

\newtheorem{corollary}[proposition]{Corollary}

\newtheorem{example}[proposition]{Example}

\begin{document}

\title{Existence $h$--principle for Engel structures}

\author{Roger Casals}
\address{Instituto de Ciencias Matem\'aticas CSIC-UAM-UC3M-UCM, C. Nicol\'as Cabrera, 13-15, 28049 Madrid, Spain}
\email{casals.roger@icmat.es}

\author{Jos\'e Luis P\'erez}
\address{Instituto de Ciencias Matem\'aticas CSIC-UAM-UC3M-UCM, C. Nicol\'as Cabrera, 13-15, 28049 Madrid, Spain}
\email{joseluis.perez@icmat.es}

\author{\'Alvaro del Pino}
\address{Instituto de Ciencias Matem\'aticas CSIC-UAM-UC3M-UCM, C. Nicol\'as Cabrera, 13-15, 28049 Madrid, Spain}
\email{alvaro.delpino@icmat.es}

\author{Francisco Presas}
\address{Instituto de Ciencias Matem\'aticas CSIC-UAM-UC3M-UCM, C. Nicol\'as Cabrera, 13-15, 28049 Madrid, Spain}
\email{fpresas@icmat.es}

\subjclass[2010]{Primary: 53A40, 53D35.}
\date{\today}

\begin{abstract}
In this article we prove that the inclusion of the space of Engel structures of a smooth $4$--fold into the space of full flags of its tangent bundle induces surjections in all homotopy groups. In particular, we construct Engel structures representing any given full flag.
\end{abstract}

\maketitle

\section{Introduction}\label{sec:intro}
In this section we state the main result of the article, an existence h--principle for Engel structures. Let us first introduce the relevant mathematical context for this result and the necessary definitions in order to give a precise statement.

\subsection{Motivation}

Let $M$ be a smooth $n$--dimensional manifold. An $m$--distribution on $M$ is a smooth correspondence assigning an $m$--dimensional subspace in $T_pM$ to each point $p \in M$ -- equivalently, it is a smooth section of the grassmanian bundle $\mbox{Gr}_m(TM)\longrightarrow M$. The group of diffeomorphisms of the manifold acts in the space of distributions by push--forward, and E. Cartan studied conditions for this action to be locally transitive at the level of germs \cite{Car,Mo}. The distributions for which this action is locally transitive are called topologically stable, and he proved that such distributions necessarily conform to the inequality
$$  m (n-m) \leq n.$$

E.~Cartan's theorem states that there are only four classes of distributions that form an open subset in the space of distributions and are topologically stable. These are line fields, contact structures, even--contact structures and Engel structures (see \cite{VG} for a more detailed account). The common feature of these distributions is that they do not possess local invariants. The existence or non--existence of geometric global invariants is therefore a crucial factor for them to be mathematically relevant. This is the reason why the global theory of line fields -- Dynamical Systems -- is so rich. Likewise, the study of global invariants of contact structures has developed into the field of Contact Topology. That is, both line fields and contact structures posses geometric global invariants. In contrast, even--contact structures satisfy a complete $h$--principle (see \cite[Theorem 14.2.3]{EM},\cite{Gi,McD}) and hence their global structure is determined by their formal algebraic topology invariants, making them uninteresting from a Differential Geometry perspective.

Engel structures still await for an answer: it is not known whether they can classified only in terms of the underlying algebraic topology. In case the answer is affirmative, there will be no Engel geometry. If instead their classification is finer, Engel geometry should rise as a relevant area within Differential Geometry, as Contact Geometry has since the discovery of global contact invariants (see \cite{El} for a summary of the history). This article focuses on the study of Engel structures, which we now define.

Let $M$ be a smooth $4$--fold, an \textbf{Engel structure} $\SD$ is a maximally non--integrable $2$--distribution, i.e.~a 2--distribution $\SD\subseteq TM$ conforming to the following two properties:
\begin{enumerate}
\item $[ \SD, \SD] = \SE$ is a $3$--distribution,
\item $[ \SE, \SE] = TM$.
\end{enumerate}
By definition, the second condition is equivalent to the 3--distribution $\SE$ being even--contact. The kernel of the bilinear map
$$ [\quad,\quad]: \SE \times \SE \to TM/\SE$$
is a line field $\SW \subseteq \SE$ which also satisfies that $\SW \subseteq \SD$. Indeed, otherwise we would have a point $p\in M$ where $\SW_p \oplus \SD_p = \SE_p$ and therefore $[\SD_p, \SD_p] = [\SE_p, \SE_p] = T_pM$, a contradiction.

The bundle $\SE$ is naturally oriented: we can choose a (local) frame $\{X,Y\}$ for the 2--distribution $\SD$ and then the ordered set $\{X,Y,[X,Y]\}$ induces an orientation of $\SE$ which is independent of any choice.

Then we define a \textbf{formal Engel structure} to be a full flag $\SW^1 \subseteq \SD^2 \subseteq \SE^3 \subseteq TM$ in the tangent space together with a fixed orientation on $\SE$. Such a flag is denoted by $(\SW, \SD, \SE)$, and the chosen orientation of $\SE$ is implied.

Consider an oriented Engel distribution $\SD^2$ in an oriented 4--fold $M$. Then $M$ is parallelizable and the flag $\SW^1 \subseteq \SD^2 \subseteq \SE^3 \subseteq TM$ induced by the Engel structure $\SD$ is both oriented and co--oriented.

It was proven by T.~Vogel \cite{Vo} that the necessary condition of $M$ being parallelizable is also sufficient for it to admit an Engel structure. His proof deeply uses the interaction between Engel structures and contact structures. In brief, a result of D.~Asimov states that a parallelizable manifold admits a round handle decomposition, and T.~Vogel's argument uses this to reduce the existence of an Engel structure to the construction of appropriate Engel structures with contact boundary on every round handle.

T.~Vogel's result provides a partial answer to the question posed in \cite[Intrigue F2]{EM}: do Engel structures in closed parallelizable $4$--folds satisfy an $h$--principle? The theorem in \cite{Vo} implies that any parallelizable $4$--fold admits an Engel structure, yet this is still far from an existence $h$--principle: T.~Vogel's method yields limited control of the homotopy type of the full flag associated to the resulting Engel structure. This article improves the result in this direction.

\subsection{Statement of the results}
Let $M$ be a $4$--fold, and consider the space of formal Engel structures $\Flags(M)$ and the space of Engel structures $\Engel(M)$ defined as
\[ \Flags(M) = \{ (\SW^1,\SD^2,\SE^3) \text{ $|$ } \SW \subset \SD \subset \SE \subset TM \text{ is a full flag and $\SE$ is oriented } \} \]
\[ \Engel(M) = \{ (\SW^1,\SD^2,\SE^3) \text{ $|$ } \SD \text{ is Engel with } \SW \subset \SD \subset \SE \text{ its associated full flag } \} \]
and endowed with the compact--open topology; observe that there is an inclusion $i: \Engel(M) \to \Flags(M)$.

The main result of this article concerns the behaviour of the inclusion in all homotopy groups:

\begin{theorem} \label{thm:main}
The map $\pi_k(i): \pi_k\Engel(M)\longrightarrow\pi_k\Flags(M)$ is surjective for every $k \geq 0$. In particular, every formal Engel structure is homotopic to the flag of a genuine Engel structure.
\end{theorem}

This strongly generalises T.~Vogel's result \cite{Vo}, which in these terms states that the set $\pi_0(\Engel(M))$ is not empty as soon as $M$ admits a completely oriented full flag. It is significant to note that T.~Vogel's argument uses flexibility in contact geometry to make the problem flexible, whereas our approach stays as far as possible from contact geometry. In addition, our construction does not use the assumption of parallelizability: the only natural orientation appearing from an Engel structure is that of the 3--distribution $\SE$, and this is the one that we control.

The methods used in the proof of Theorem \ref{thm:main} imply the following

\begin{corollary} \label{cor:cobordism}
Let $(M,\partial M)$ be a 4--fold with boundary and $(\SW,\SD,\SE)$ a formal Engel structure such that the line field $\SW$ is transverse to the boundary $\partial M$ and $(\partial M, T(\partial M) \cap \SE)$ is a contact structure. Then there is a deformation of the formal Engel structure into a genuine Engel structure inducing the same contact structure on the boundary $\partial M$.
\end{corollary}

In particular, this implies that the notion of Engel cobordism or Engel fillability is not particularly relevant in order to distinguish contact manifolds. In \cite{Vo}, T.~Vogel constructed an Engel structure on the cobordism $\mathbb{S}^2 \times \mathbb{S}^1 \times [0,1]$ with one boundary component being a tight contact structure and the other being overtwisted. Corollary \ref{cor:cobordism} states that given any $3$--fold $V$ and two homotopic contact structures $\xi_0$, $\xi_1$ (which admit a global legendrian line field), there exists an Engel structure on the trivial cobordism $V \times [0,1]$ inducing the contact manifold $(V\times\{0\},\xi_0)$ and $(V\times\{1\},\xi_1)$ on the boundary components.

In view of Theorem \ref{thm:main}, a reasonable question is whether the map $i$ is also injective on $\pi_k$, possibly after restricting to some subclass within $\Engel(M)$. The method of proof for Theorem \ref{thm:main} allows the careful reader to guess possible definitions for that class. This will be the content of future work. The existence of a proper class satisfying a complete h--principle would start Engel topology as a meaningful area within Differential Topology.

Let us consider a second corollary from our main theorem. Consider a $4$--dimensional foliation $\SF$ in a smooth $n$--dimensional manifold $M$. Then a flag $\SW^1 \subseteq \SD^2 \subseteq \SE^3 \subseteq T\SF$ and an orientation of $\SE^3$ are said to be a formal foliated Engel structure for the foliation $\SF$; a $2$--distribution $\SD \subseteq T\SF$ is called a foliated Engel structure if it is an Engel structure when restricted to each leaf of the foliation $\SF$. Denote the spaces of formal foliated Engel structures and foliated Engel structures by $\Flags(\SF)$ and $\Engel(\SF)$, respectively. The parametric nature of Theorem \ref{thm:main} implies the following result.
\begin{theorem} \label{cor:folia}
The inclusion map $\pi_k(i):\pi_k(\Engel(\SF))\longrightarrow\pi_k(\Flags(\SF))$ is surjective for every $k \geq 0$. Hence, formal foliated Engel structures are homotopic to flags of genuine foliated Engel structures.
\end{theorem}

\subsection{Structure of the paper}

The article is organized as follows. In Section \ref{sec:definitions} we define all the objects involved and we discuss some known results. Subsection \ref{ssec:extensions} is classical to an extent, since it can be mainly found in the works of E. Cartan \cite{Car2}, though it has been overlooked for many years. It can be condensed into Proposition \ref{prop:EngelCurves}, which is a fundamental ingredient in the proof of Theorem \ref{thm:main}.

The article primarily focuses in the proof of the $\pi_0$--statement of Theorem \ref{thm:main}; the argument in this case is central for the remaining results, and once described in detail it can be readily applied to the parametric case in order to prove the $\pi_k$--statements.

The proof of the $\pi_0$--statement of Theorem \ref{thm:main} consists of three parts. First, given some full flag with oriented 3--distribution $\SE$ we turn it into a flag whose $3$--distribution is an even contact structure and whose $1$--distribution is its kernel. This is achieved with the $h$--principle for even--contact structures.

Second, in Section \ref{sec:red}, we triangulate $M$ in a manner adapted to the kernel of the formal Engel structure, and subsequently deform the formal Engel structure to a genuine Engel structure in a neighbourhood of the $3$--skeleton. This reduction process provides a collection of $4$--cells endowed with formal Engel structures that are genuine Engel structures in the boundary. We also prove in this section that such formal Engel structures on the 4--cells can be assumed to be of a particular form, which we call the \textbf{$6\pi$--radial shells}.

Third, in Section \ref{sec:universalHole}, we construct an object called the \textbf{four--leaf clover} which allows us to deform any $6\pi$--radial shell into a genuine Engel structure, thus concluding the proof of the $\pi_0$--statement in Theorem \ref{thm:main}. Note that the argument is not presented in a linear fashion, and the chosen order serves to better motivate the constructions.

Section \ref{sec:proof} discusses the parametric nature of the construction, and concludes Theorem \ref{thm:main}, Theorem \ref{cor:folia} and Corollary \ref{cor:cobordism} from the results in the proof of the $\pi_0$--statement of Theorem \ref{thm:main}.

\subsection{Acknowledgements} We are grateful to V.~Colin, V.L.~Ginzburg, E.~Giroux, E.~Murphy and A.~Stipsicz for useful discussions. We would like to especially acknowledge Y.~Eliashberg and T.~Vogel for intense and valuable discussions during the conference h--Principles in Houat, the arguments in this article have been greatly simplified thanks to them. The classical construction explained in Example \ref{ex:lorentzian} was pointed out to us by Daniel Fox and it has been an important intuition for the development of this work. The authors are supported by Spanish National Research Project MTM2013---42135. This work is supported in part by the ICMAT Severo Ochoa grant SEV-2011-0087 through the V. Ginzburg Lab. \'A.~del Pino is supported by La Caixa--Severo Ochoa grant. J. L.~P\'erez is supported by a MINECO FPI grant.

\section{Preliminaries on Engel Structures} \label{sec:definitions}

In this section we discuss even--contact structures and describe two relevant properties which are used in the argument of Theorem \ref{thm:main}. Second, we detail the characterization of certain Engel structures on the 4--cell $\D^3\times[0,1]$ in terms of curves in the 2--sphere $\NS^2$, and provide two valuable examples of Engel structures due to E.~Cartan \cite{Car2}.

\subsection{Even--contact structures}

In the previous section we briefly introduced even--contact structures. Let us describe this notion in detail.

\begin{definition}
Let $M$ be a smooth $(2n+2)$--dimensional manifold. A $(2n+1)$--distribution $\SE^{2n+1} \subseteq TM$ is said to be an even--contact structure if it is locally described by a 1--form $\alpha$ such that $\alpha \wedge (d\alpha)^n \neq 0$. $(M, \SE)$ is called an even--contact manifold, and $\SE$ is said to be coorientable if $\alpha$ can be defined globally.
\end{definition}

Even--contact structures can be regarded as transverse contact structures associated to line fields: the even--contact condition amounts to the 2--form $d\alpha$ being of maximal rank in $\ker(\alpha)$, and the kernel $\SW$ of the 2--form $d\alpha|_{\ker(\alpha)}$ is a real line field. They satisfy the following

\begin{lemma}
Let $(M^{2n+2}, \SE)$ be an even--contact manifold and $N^{2n+1} \subseteq M$ a $(2n+1)$--dimensional submanifold transverse to the kernel $\SW$ of $\SE$. Then $(N, \SE \cap TN)$ is a contact manifold.
\end{lemma}
% \vspace{-0.3cm}
\begin{proof}
Given a locally defining 1--form $\alpha$ for the even--contact $\SE$, the $(2n+1)$--form $\alpha \wedge (d\alpha)^n|_N$ is a volume form on $N$ since the kernel $\SW$ of the 2--form $d\alpha$ is transverse to $N$.
\end{proof}
% \vspace{-0.2cm}

Given that even--contact structures $(M^4,\SE)$ induce contact structures in $3$--folds $N$ transverse to their kernel $\SW$, so do Engel structures $(M^4,\SD)$. In addition, the line field $TN \cap \SD \subseteq TN \cap \SE$ is a distinguished Legendrian vector field in the contact manifolds $(N,\SE\cap TN)$.

There are however two significant differences between contact and even--contact structures. A first invariant associated to an even--contact structure is its kernel: this line field often has complicated dynamics, and these are unstable under smooth perturbations of the even--contact structure. Therefore, Gray's stability cannot hold in full generality (although it does hold if we fix the line field $\SW$ \cite{Go}). Note that there is a Darboux theorem that, following the previous Lemma, states that even--contact structures are locally isomorphic to the stabilisation by $\mathbb{R}$ of the contact Darboux normal form.

The second difference is the existence of global geometric invariants. Even--contact structures satisfy a complete h--principle regarding their classification: pairs $(\alpha, \omega)\in\Omega^1(M)\times\Omega^2(M)$ where $\omega$ has maximal rank on $\ker(\alpha)$ are referred to as almost even--contact structures, and D.~McDuff used convex integration techniques to prove the following

\begin{theorem} \emph{(\cite[Section 20.6]{EM},\cite{McD})}\label{thm:hprincipleEvenContact}
For any given smooth manifold $M^{2n+2}$, there is a weak homotopy equivalence induced by the inclusion between the space of even contact structures and the space of almost even contact structures.
\end{theorem}

Theorem \ref{thm:hprincipleEvenContact} is used in the proof of Theorem \ref{thm:red}, the main reduction result used in Theorem \ref{thm:main}. This reduces the construction of an Engel structure from a formal Engel structure to the construction of an Engel structure from an even--contact structure.

Recall that M.~Gromov's $h$--principle applies to Engel structures in open manifolds. In particular, for an open 4--fold, the space of Engel structures is weakly homotopy equivalent to the space of formal Engel structures (see \cite[Theorem 7.2.3]{EM}). The contribution of Theorem \ref{thm:main} is a surjection h--principle for Engel structures in closed 4--folds.

\subsection{Engel structures in $\mathbb{D}^3 \times [0,1]$} \label{ssec:extensions}

In this subsection we discuss the relation between Engel structures and families of convex curves in the 2--sphere $\NS^2$. This will be used in Section \ref{sec:universalHole} in order to construct a genuine Engel structure in a 4--cell with appropriate boundary conditions.

\subsubsection{Engel structures as curves}

Consider the 4--cell $\mathbb{D}^3 \times [0,1]$ with coordinates $(x,y,z,t)$. Let $\SD=\langle\partial_t,X\rangle$ be a $2$--plane distribution where the vector field $X$ is tangent to the foliation by level sets $\mathbb{D}^3 \times \{t_0\}$, $t_0\in[0,1]$. Let us write $\overset{.}{X} = [\partial_t, X]$ and $\overset{..}{X} = [\partial_t, \overset{.}{X}]$, which are two vector fields also tangent to these level sets.

The three vector fields $X$, $\overset{.}{X}$ and $\overset{..}{X}$ on $D^3\times[0,1]$ can be regarded as $1$--parametric families of vector fields in $\mathbb{D}^3$, with parameter $t \in [0,1]$; these families are denoted by $X_t$, $\overset{.}{X}_t$ and $\overset{..}{X}_t$.

Trivialize $T\mathbb{D}^3$ with the coordinate frame $\langle \partial_x,\partial_y,\partial_z\rangle$ to identify all the fibres of the 2--sphere bundle $\mathbb{S}(T\mathbb{D}^3)$ with a fixed 2--sphere $\mathbb{S}^2$. For $p \in \mathbb{D}^3$ fixed, consider the restriction $X_p$ of the vector field $X$ to the vertical segment $\{p\} \times [0,1]$. Then the vector field $X_p$ describes a curve in $\mathbb{S}^2$ and thus the $2$--plane distribution $\SD$ is given by a $\mathbb{D}^3$--family of such curves.

In order to characterize the Engel condition from this viewpoint, we briefly discuss convex curves.

\subsubsection{Convex curves in $\NS^2$} Consider a parametrized smooth curve $\gamma:[0,1]\longrightarrow \NS^2$. Its unit tangent vector field is given by $\mathfrak{t}(t)=\gamma'(t)/||\gamma'(t)||$. Define $\mathfrak{n}(t)$ to be the unique vector field such that $\{\mathfrak{t}(t), \mathfrak{n}(t)\}$ is an orthonormal oriented basis of the tangent space $T_{\gamma(t)} \NS^2$. A point $\gamma(t)$ is said to be an \textbf{inflection point} of the curve $\gamma$ if $\langle\mathfrak{t}'(t),\mathfrak{n}(t)\rangle=0$, and the curve $\gamma$ is said to be \textbf{convex} if it has no inflection points.

The significance of this condition in terms of the Engel structure will shortly be explained. The following result proves that the homotopy classification of convex curves in the 2--sphere is determined by the homotopy class of the Frenet map
$$\mathfrak{F}(\gamma): [0,1] \to SO(3),\quad\mathfrak{F}(\gamma)(t) = (\gamma(t), \mathfrak{t}(t), \mathfrak{n}(t)).$$

\begin{theorem}[\cite{Lit}]\label{thm:Little}
The connected components of the space of convex closed curves in $\NS^2$ are
\begin{wideitemize}
  \item[1.] curves with $[\FF(\gamma)] \in \pi_1(SO(3))$ trivial,
  \item[2.] embedded curves with $[\FF(\gamma)]$ non trivial,
  \item[3.] curves that are not embbeded with $[\FF(\gamma)]$ non trivial.
\end{wideitemize}
\end{theorem}

In Section \ref{sec:universalHole} we consider curves that are not convex but that fail to be so in an explicit manner. These curves are $C^\infty$--limits of convex curves that become increasingly tangent to the equator $\{z=0\} \subseteq \NS^2$. In this case we can define the Frenet map of such a curve as the limit of the Frenet maps of the convex curves approaching it. 

\subsubsection{The Engel condition}

Following the description of Engel structures $\SD=\langle \partial_t,X\rangle$ on the 4-cell $\D^3\times[0,1]$ in terms of families of curves on the 2--sphere, we now provide a sufficient condition for these families to define Engel structures.

\begin{proposition} \label{prop:EngelCurves}
A $2$--distribution $\mathcal{D}=\langle \partial_t,X\rangle$ is an Engel structure at a point $(p,t) \in \mathbb{D}^3 \times [0,1]$ if both $\overset{.}X(p,t) \neq 0$ and at least one of the following two conditions holds:
\begin{wideitemize}
  \item[1.] the curve $X_p:[0,1]\longrightarrow\mathbb{S}^2$ has no inflection point at time $t$,
  \item[2.] the 2--distribution $\langle X_t, \overset{.}{X}_t \rangle$ is a contact structure on $\Op(p)\times\{t\}\subseteq\D^3\times\{t\}$.
\end{wideitemize}
\end{proposition}
\begin{proof}
First, the 2--distribution $\SD$ not being integrable translates to the condition $\overset{.}X(p,t) \neq 0$. Indeed, since the vector field $[\partial_t,X]$ is tangent to the foliation by level sets, the condition for the vector $[\partial_t,X]_{(p,t)}$ not being in $\SD$ is equivalent to the vector field $\overset{.}{X}_t$ not being colinear with $X_t$ at the point $p$, so the curve $X_p$ has non--zero velocity $\overset{.}X(p,t) \neq 0$ at the point $(p,t)$.

Set $\SE = \langle \partial_t, X, \overset{.}{X} \rangle$. For the 2--distribution $\SD$ to be an Engel structure, the 3--distribution $\SE$ must be non--integrable, so at least one of the two vectors $\overset{..}{X} = [\partial_t, \overset{.}{X}]$ and $[X, \overset{.}{X}]$ should not be contained in $\SE$ at the point $(p,t)$. If $t\in[0,1]$ is not an inflection point of the curve $X_p$, the acceleration $\overset{..}{X}_t$ is not contained in the space spanned by the position $X_t$ and the speed $\overset{.}{X}_t$ at $p$. If the 2--distribution $\langle X_t, \overset{.}{X}_t \rangle$ is a contact structure on the level $\Op(p)\times\{t\}$, the Lie bracket satisfies $[X_t, \overset{.}{X}_t] \notin \langle X_t, \overset{.}{X}_t \rangle$ at the point $p$. 
\end{proof}

Proposition \ref{prop:EngelCurves} generalizes two classical constructions due to E.~Cartan \cite{Car2}:

\begin{example}[contact prolongation]
Let $(N,\xi)$ be a contact 3--fold and $\xi=\langle Y, Z\rangle$ a frame. The contact prolongation of $(N,\xi)$ is the Engel structure $(N \times [0,1],\SD)$ defined by
$$\SD(p,t)= \langle \partial_t, X(p,t)= \cos(t)Y(p)+ \sin(t)Z(p) \rangle.$$
In this case, Condition $(2)$ in Proposition \ref{prop:EngelCurves} is satisfied, which proves that $\SD$ is an Engel structure.
\end{example}

\begin{example}[lorentzian prolongation] \label{ex:lorentzian}
Consider a Lorentzian 3--fold $(N,g)$ with a type $(1,2)$ framing $\langle L^+, Y_-, Z_-\rangle$. The kernels of the Lorentzian metric at each point define a family of cones on the tangent bundle $TN$ which (after trivialization with the framing) provide a family of non--degenerate quadric curves $C_p$ in the unit 2--sphere $\NS^2$. Parametrize each curve $X_p: [0,1]\longrightarrow C \subseteq \NS^2$ and define the lorentzian prolongation $(N \times [0,1],\SD)$ as the $2$--distribution defined by $$\SD(p,t)= \langle \partial_t, X_p(t) \rangle.$$
This is an Engel structure because Condition $(1)$ of Proposition \ref{prop:EngelCurves} is satisfied.
\end{example}

Notice that the Legendrian line field $\SW$ is expanded by $\partial_t$ in the case of the contact prolongation and it is transverse to the direction $\partial_t$ in the lorentzian prolongation. The combination of these two constructions thus requires a deformation of the dynamics of the line field $\SW$, and this allows to create flexibility in the space of Engel structures. Proposition \ref{prop:EngelCurves} is a crucial ingredient for the extension result stated in Theorem \ref{thm:fill}, which is one of the two parts for the argument of Theorem \ref{thm:main}.

Proposition \ref{prop:EngelCurves} and Example \ref{ex:lorentzian} prove that convexity of the corresponding $\D^3$--family of curves implies that the $2$--distribution $\SD=\langle \partial_t,X_p(t)\rangle$ is an Engel structure on the 4--cell $\mathbb{D}^3\times[0,1]$.

% ------------------------------
% ------------------------------
\section{The Hole and Its Filling}\label{sec:universalHole}

In this section we address the problem of extending a particular germ of Engel structure on $\Op(\partial \D^4)$ to an Engel structure in the interior of $\D^4$. The reduction process explained in Section \ref{sec:red}, subsumed in Theorem \ref{thm:red}, implies that such an extension suffices in order to prove Theorem \ref{thm:main}.

Subsection \ref{ssec:engel_shell} introduces in detail this extension problem and Subsection \ref{ssec:radial} relates different extension problems in order to obtain a simpler model. Subsection \ref{ssec:comb} provides a useful rephrasing in terms of curves. In Subsection \ref{ssec:clover} we explain the solution up to three technical lemmas, whose statement and proof we defer to Subsection \ref{ssec:Glueing}. The influence of the article \cite{BEM} is manifest in this section.

% ------------------------------
\subsection{Engel shells}\label{ssec:engel_shell}

The following definition describes an Engel germ in the boundary $\partial(\D^3 \times [0,1])$ of the 4--cell $\D^3\times[0,1]$ that extends to the interior as a formal Engel structure. Consider coordinates $(x,y,z;t)$ in the cartesian product $\D^3 \times [0,1]$.

\begin{definition} \label{def:engelShell}
An \textbf{Engel shell} is a formal Engel structure $(\SW,\SD,\SE)$ on the 4--cell $\D^3\times[0,1]$ conforming to the following properties:
\begin{wideitemize}
\item[1.] $\SD = \langle \partial_t, X \rangle$, where $X$ is tangent to the level sets $\D^3\times\{t\}$,
\item[2.] In a neighbourhood $\Op(\partial(\D^3 \times [0,1]))$ of the boundary:

\begin{wideitemize}
\item[a.] The 2--distribution $\SD$ is an Engel structure,
\item[b.] $\SE = \xi \oplus \partial_t$, with $\xi$ a $t$--invariant contact structure on the level sets $\D^3\times\{t\}$,
\item[c.] $\SW = \langle \partial_t \rangle$ and $X$ is tangent to the 2--distribution $\xi$,
\item[d.] $\{\partial_t,X,[\partial_t,X] \}$ is a positive frame for $\SE$.
\end{wideitemize}
\end{wideitemize}
In case $(\SW,\SD,\SE)$ defines an Engel structure on $\D^3\times[0,1]$, the Engel shell is said to be solid.
\end{definition}

Let us discuss the homotopic properties of formal Engel structures with fixed Engel structure in the boundary.

%Given the 2--distribution $\SD=\langle\partial_t,X\rangle$ on a neighbourhood $\Op(\partial(\D^3 \times [0,1]))$ of the boundary, there can be at most two extensions (up to homotopy) of the 2--distribution $\SD$ to a 2--distribution of the form $\langle\partial_t,X\rangle$ on the interior $\D^3\times[0,1]$. The reason being that this extension is tantamount to the extension of a non--zero vector field $X$ tangent to the level sets $\D^3\times\{t\}$ and the space of sections of an $\S^2$--bundle over the pair $(\D^4,\partial\D^4)$ has two connected components corresponding to $H^4(\D^4,\partial;\pi_4(\S^2))\cong\mathbb{Z}_2$.

Suppose that the 2--distribution $\SD=\langle\partial_t,X\rangle$ is extended to the interior. The extension of the 3--distribution $\SE$ to the interior is equivalent to the extension of a given vector field $V\subseteq\SE$ on the boundary, normal to the 2--distribution $\SD\subseteq\SE$, to a vector field on the interior also normal to the 2--distribution $\SD$. The space of such vector fields is diffeomorphic to the space of sections of a circle bundle over the pair $(\D^4,\partial \D^4)$, which is a non--empty contractible space. A similar reasoning implies that the space of extensions of $\SW$ to the interior once it has been fixed along a neighbourhood of the boundary is non--empty and contractible.

In consequence, the homotopic properties of the formal Engel structure are determined by the homotopy class of the 2--distribution $\SD$, which at the same time is determined by the homotopy class of the line field $X$. This justifies the notation $\SD$ or also $X$ for an Engel shell $(\D^3\times[0,1],\SW,\SD,\SE)$.

\subsection{Angular shells}\label{ssec:angular_shell} Engel shells are more general than those resulting from the reduction process stated in Theorem \ref{thm:red}. The main reason is that the first step in the reduction process deforms the given formal Engel structure to a formal Engel structure in which the 3--distribution $\SE$ is even--contact and the line field $\SW$ is its kernel. In particular, we can reduce the extension problem for Engel shells to those in which
$$\SW = \langle \partial_t \rangle,\quad \SE = \xi \oplus \SW,\quad X \in \xi \times \{Êt \}Ê\subset  T(\D^3\times\{t\})$$
not only on the boundary $\Op(\partial(\D^3 \times [0,1]))$, but on the interior $\D^3\times[0,1]$. This particular type of Engel shells $\SD$ are called \textbf{angular shells}. The advantage of an angular shell is that it can be described by one real--valued function; we now explain this.

Consider the euclidean metric in $\D^3\times[0,1]$. Let $\SD=\langle\partial_t,X\rangle$ be an angular shell and assume that the vector field $X$ is unitary. Fix also an orthonormal Legendrian frame $\{Y,Z\}$ for the contact structure $(\D^3, \xi)$ such that $\{\partial_t,Y,Z\}$ is a positive frame for the 3--distribution $\SE$. This choice assigns to each angular shell a real--valued function $c : \mathbb{D}^3 \times [0,1]\longrightarrow\R$
\begin{equation} \label{eq:angleCorrespondence} X(p,t)= \cos(c(p,t))Y + \sin(c(p,t))Z, \end{equation}
which is uniquely defined up to shifting by $2\pi$.

Given an angular shell $\SD$, the function $c = c(\SD)$ defined by Equation \ref{eq:angleCorrespondence} is called its \textbf{angle function}. The discussion on Subsection \ref{ssec:extensions} and the orientation conventions imply the following fact:

\begin{lemma}\label{lem:positive}
The angular shell $\SD$ is an Engel structure at the point $(p,t)$ if and only if $\partial_tc(\SD)(p,t) > 0$.
\end{lemma}

In particular we have the differential inequality $\partial_t c(\SD) > 0$ on a neighbourhood $\Op(\partial(\D^3\times[0,1]))$.

Conversely, suppose that a function $c: \D^3\times[0,1] \longrightarrow \R$ satisfies $\partial_t c(\SD) > 0$ on a neighbourhood $\Op(\partial(\D^3\times[0,1]))$. Then $c$ is the angle function of some angular model $\SD(c)$ which is uniquely defined.

In consequence, there is a bijective correspondence between angle functions up to shifting by $2\pi$ and angular models.
% A useful concept to define is that of \textbf{total angle} associated to an angle function: $\tilde c(p) = c(p,1) - c(p,0)$.
Contractibility of the space of real functions relative to the boundary implies that:

\begin{lemma}\label{lem:homotopyAngleFunction}
The angular shells $\SD(c_1)$ and $\SD(c_2)$ are homotopic relative to the boundary as angular shells if and only if their angle functions $c_1,c_2:\D^3\times[0,1]\longrightarrow\R$ agree on $\Op(\partial(\D^3\times[0,1]))$.
\end{lemma}

The following example illustrates the simplest case in which an angular shell can be homotoped to a genuine Engel structure. Bolzano's theorem shows that, in general, the extension problem is obstructed if one tries to solve it within the space of angular shells.

\begin{example}
Suppose that $c(p,1)>c(p,0)$ for all $p \in \D^3$. Lemmas \ref{lem:positive} and \ref{lem:homotopyAngleFunction} imply that the angular shell $\SD(c)$ is homotopic relative to the boundary to a solid Engel shell on $\D^3\times[0,1]$.
\end{example}

\subsection{Domination and radial shells} \label{ssec:radial} The extension problem for the germ of an Engel structure in the boundary of $\D^3\times[0,1]$ to its interior introduces a partial order between angular shells.

\begin{definition}
Let $\SD(c_1)$ and $\SD(c_2)$ be two angular shells, $\SD(c_1)$ \textbf{dominates} $\SD(c_2)$ if
\[ c_1(p,0) \leq c_2(p,0) \text{ and } c_2(p,1) \leq c_1(p,1). \]
\end{definition}

The following proposition reduces the problem of filling angular shells to filling angular shells with simple angular functions presenting some symmetry.

\begin{proposition} \label{prop:domination}
Let $\SD(c_1)$ and $\SD(c_2)$ be two angular shells such that $\SD(c_1)$ dominates $\SD(c_2)$. If $\SD(c_2)$ admits a deformation to a solid Engel shell through Engel shells, then so does $\SD(c_1)$.
\end{proposition}
\begin{proof}
There are smooth functions $h_1,h_2: \D^3 \to [0,1]$ such that $c_1(p,h_i(p)) = c_2(p,i)$, $i=0,1$. Use Lemma \ref{lem:positive} to deform $c_1$ to be strictly increasing for $t \in \Op([0,h_0(p)]) \cup \Op([h_1(p),1])$.

Now there is a unique embedding $\Phi: \Op(\partial \SD(c_2)) \to \SD(c_1)$ satisfying $\Phi^*c_1 = c_2$. Extending $\Phi$ to the interior of $\SD(c_2)$ arbitrarily, $\SD(c_1)$ can be homotoped within $\Phi(\SD(c_2))$, relative to its boundary, to achieve $\Phi^*c_1 = c_2$. Perform the Engel deformation in $\Phi(\SD(c_2))$ provided by assumption. The claim follows.
%First deform the angle function $c_1$ on $[0,1/7]\cup[6/7,1]$ to obtain an angle function $c'_1$ such that
%\begin{wideitemize}
%\item[-] $c'_1(p,t)=c_1(p,t)$ on $[0,1/9]\cup[1/7,6/7]\cup[8/9,1],$
%\item[-] $c'_1(p,1/8)=c_2(p,0)$ and $c'_1(p,7/8)=c_2(p,1),$
%\item[-] $\partial_tc'_1(p,t)>0$ on $[0,1/8]\cup[7/8,1].$
%\end{wideitemize}
%
%Note that we can deform the initial Engel shell $\SD(c_1)$ to the Engel shell $\SD(c_1')$ and thus it suffices to deform $\SD(c'_1)$ to a solid Engel shell. Consider the unique embedding $\Phi: \Op(\partial \SD(c_2))\longrightarrow\SD(c'_1)$ satisfying $\Phi^*c'_1 = c_2$. Extending $\Phi$ to the interior of $\SD(c_2)$ arbitrarily, the 2--distribution $\SD(c_1')$ can be homotoped within $\Phi(\SD(c_2))$, relative to its boundary, to achieve $\Phi^*c'_1 = c_2$ in the interior of the Engel shell $\SD(c_2)$. The required deformation is obtained by performing the Engel deformation in $\Phi(\SD(c_2))$ provided by hypothesis.
\end{proof}

Given an angle function $c:\D^3\times[0,1]\longrightarrow\R$ and a point $p\in\D^3$, the difference $c(p,1)-c(p,0)$ measures the amount of rotation of the Legendrian vector field $X$. The extension problem that we solve in Theorem \ref{thm:fill} concerns a particular class of angular shells: these are angular shells which rotate enough along each vertical segment of the boundary $\partial\D^3\times[0,1]$. In order to describe in precise terms this geometric intuition, we introduce the following definition. Hereafter, the symbol $\rho$ denotes a fixed numeric real value such that $[\rho,2\rho]\subset (0,1)$.

\begin{definition}
Let $K\in\R^+$ be a constant. A function $c:\D^3\times[0,1]\longrightarrow\R$ is said to be \textbf{$K$--radial} if it conforms to the following three properties
\begin{wideitemize}
\item[a.] $c(p,t)$ is increasing in $t\in[0,2\rho]$ and satisfies
$$c(p,t) = c(p,\rho) + \frac{(t-\rho)}{\rho}\cdot K,\mbox{ for }t \in [\rho,2\rho],$$
\item[b.] $c(p,t) - c(p,\rho)$ is invariant under the action of $SO(3)$ on $\D^3$,
\item[c.] $c(p,t)$ is $p$--invariant for $(p,t) \in \Op(\{0\}\times[0,1])$,
\end{wideitemize}
The angular shell $\SD(c)$ associated to a $K$--radial angle function $c$ is said to be a \textbf{$K$--radial shell}.
\end{definition}

In this notation, the most relevant use of Proposition \ref{prop:domination} is the following corollary.

\begin{corollary} \label{cor:dominationSymmetric}
Let $c$ be an angle function, and $K\in \R^+$ such that
$$K < \min_{p\in\partial \D^3}(c(p,1)-c(p,0)).$$
Then there exists a $K$--radial function $c'$ such that $\SD(c)$ dominates $\SD(c')$.
\end{corollary}

It is time to state Theorem \ref{thm:fill}, the main result of this section and, along with Theorem \ref{thm:red}, one of the two key ingredients in the proof of the existence h--principle stated in Theorem \ref{thm:main}. The rest of this section is dedicated to its proof, and its statement reads as follows:

\begin{theorem} \label{thm:fill}
A $6\pi$--radial shell is homotopic through Engel shells to a solid Engel shell.
\end{theorem}

Theorem \ref{thm:red} in the next section implies that in order to prove Theorem \ref{thm:main}, it suffices to deform angular shells with difference angle $c(p,1)-c(p,0)$ greater than $6\pi$, for any $p \in \Op(\D^3)$, to solid Engel shells. In consequence, Corollary \ref{cor:dominationSymmetric} and Theorem \ref{thm:fill} indeed conclude the proof of Theorem \ref{thm:main}.

The proof of Theorem \ref{thm:fill} is essentially contained in Figure \ref{fig:movie} and it features the four--leaf clover curve as a crucial ingredient. The fact that the contact 2--plane field $\xi\subseteq T\D^3\times\{t\}$ in an angular shell cuts the unit sphere $\S^2\cong T_p\D^3\times\{t\}$ in an equator depending on the point $p$, and the essential role of the inflection points of the curves on this 2--sphere, require an additional technicality that we now address by defining Engel combs.

\subsection{Engel combs} \label{ssec:comb} Subsection \ref{ssec:extensions} implies that Engel shells can be described in terms of  $\D^3$--parametric families of parametrized curves in the 2--sphere $\NS^2$. These curves are given by the unitary vector field $X$ that determines the 2--distribution $\SD=\langle\partial_t,X\rangle$. The vector field $X$ is a section of the unit tangent bundle of the level sets $\D^3\times\{t\}$ and thus (once this bundle is trivialized) can be considered as a map $X:\D^3\times[0,1]\longrightarrow\S^2$.

Instead of the trivialization $\langle\partial_x,\partial_y,\partial_z\rangle$ provided by the coordinates, we trivialize the tangent bundles $T(\D^3 \times \{t\})$ of the level sets in a manner more suited to such families of curves. This is done as follows: for each point $(p,t) \in \D^3 \times [0,1]$, consider the $t$--invariant orientation--preserving linear isometry $\varphi_{(p,t)}: T_{(p,t)}(\D^3 \times \{t\}) \longrightarrow \R^3$ defined by the conditions
$$\varphi_{(p,t)}(Y(p,t)) = \partial_x, \quad \varphi_{(p,t)}(Z(p,t))= \partial_y,$$
where $\{Y,Z\}$ is a frame for the contact structure $(\D^3,\xi)$.
The isometries $\varphi_{(p,t)}$ identify the unit sphere of the contact plane $T_{(p,t)}(\D^3 \times \{t\}) \cap \SE_{(p,t)}$ with the horizontal equator $\mathbb{S}^2 \cap \{ z=0 \}$.

Consider the following rotation of angle $\theta$ around the $z$--axis:
$$\Rot(\theta)=\left(
\begin{array}{ccc}
\cos(\theta)&-\sin(\theta)&0\\
\sin(\theta)&\cos(\theta)&0\\
0&0&1
\end{array}
\right),\mbox{ and write }e_1 = \left(\begin{matrix} 1\\0\\0 \end{matrix}\right).$$
Then a given radial shell $\SD(c)$ yields a 1--parametric family of curves
\begin{equation} \label{eq:radialComb}
\left\{\begin{array}{l}
\gamma^c_r: [0,1] \longrightarrow \NS^2,\quad r \in [0,1], \\
\gamma^c_r(t) = \Rot(c(p,t)-c(p,\rho))e_1 \quad  \text{ where $p$ is any point in $\D^3$ with radius $r$.}
\end{array} \right.
\end{equation}

In the proof of Theorem \ref{thm:fill} we work with Engel shells that are not necessarily angular. This leads to the following definition.

\begin{definition} \label{def:EngelComb}
Let $K\in\R^+$ be a positive constant. A {\bf $K$--Engel comb} is a $[0,1]$--family of curves $\gamma_r: [0,1] \longrightarrow \NS^2$, $r \in [0,1]$, such that
\begin{wideitemize}
\item[a.] $\exists c:\D^3\times[0,1]\longrightarrow\R$ a $K$--radial function such that
$$\gamma_r(t) = \gamma^c_r(t),\mbox{ for }t \in \Op([0,2\rho] \cup \{1\})\mbox{ and }r \in \Op(\{1\}),$$
\item[b.] The curves $\gamma_r$ are $r$--invariant for $r\in\Op(\{0\})\subseteq[0,1]$,
\item[c.] The curves $\gamma_r$ are $C^\infty$--tangent to the horizontal equator $\{z= 0 \}$ in its inflection points.
\end{wideitemize}
\end{definition}

These families of curves neatly describe a particular type of Engel shells:

\begin{lemma} \label{lem:comb2shell}
Consider a $K$--Engel comb $\gamma_r$, and a smooth function $d: \D^3 \to \R$. Then the 2--plane distribution
$$\SD(\gamma_r) = \langle \partial_t, X(p,t)\rangle = \langle\partial_t,\varphi_{(p,t)}^{-1} \Rot(d(p))\gamma_{|p|}(t) \rangle$$
defines an Engel shell.
\end{lemma}
\begin{proof}
Condition (b) in Definition \ref{def:EngelComb} implies that $\SD(\gamma_r)$ is smooth near $\{0\} \times [0,1] \subseteq \D^3 \times [0,1]$, whereas Condition (a) recovers the boundary conditions of the Engel shell.
\end{proof}

Engel combs conform a strict subclass of Engel shells that is well suited for the extension problem. However, the resulting Engel shells are not necessarily solid due to the lack of control on either the velocities or the inflection points of the curves. The following definition includes an additional condition which guarantees that the Engel comb yields a solid Engel shell.

\begin{definition} \label{def:tameComb}
An Engel comb  $\gamma_r$ is said to be {\bf tame} if it satisfies the following two properties:
\begin{wideitemize}
\item[1.] $\gamma_r'$ is non--vanishing,
\item[2.] Consider the set $\mathcal{I}_{\gamma_r} = \{(r,t) \in [0,1]^2| \text{ $t$ is an inflection point of } \gamma_r\}$. 

For every $(r,t) \in \mathcal{I}_{\gamma_r}$, $\exists a,b\in\R^+$, $a<b$, such that
$(r,t) \in [a,b] \times \{t\} \subset \mathcal{I}_{\gamma_r}$.
\end{wideitemize}
\end{definition}

Indeed, these two conditions imply that tame Engel combs induce through Lemma \ref{lem:comb2shell} solid Engel shells.

\begin{proposition} \label{prop:solidTameComb}
The Engel shells induced from a tame Engel comb are solid Engel shells.
\end{proposition}
\begin{proof}
By Condition (1) in Definition \ref{def:tameComb}, the $2$--plane $\SD$ defined by $\gamma_r$ is non--integrable. Suppose that $(|p|,t)$ lies in the complement of $\mathcal{I}_{\gamma_r}$, then Proposition \ref{prop:EngelCurves} shows that $\SD$ is Engel at $(p,t)$.

In case $(|p|,t) \in \mathcal{I}_{\gamma_r}$, consider the interval $[a,b]$ provided by Condition (2) in Definition \ref{def:tameComb}. The set of points $(p',t)$ with $|p'| \in [a,b]$ is a region $S=\NS^2\times [a,b]\subset\D^3$ with non--empty interior. By Condition (c) in Definition \ref{def:EngelComb}, $\langle \gamma_r, \gamma_r' \rangle|_{[a,b] \times \{t\}} = \{z=0\}$, which means that the corresponding Engel shell has $\langle X, \overset{.}{X} \rangle|_{S \times \{t\}} = \xi$. The contact condition being open, we can also apply Proposition \ref{prop:EngelCurves}.
\end{proof}

\begin{figure}[h]
\centering
\includegraphics[scale=0.35]{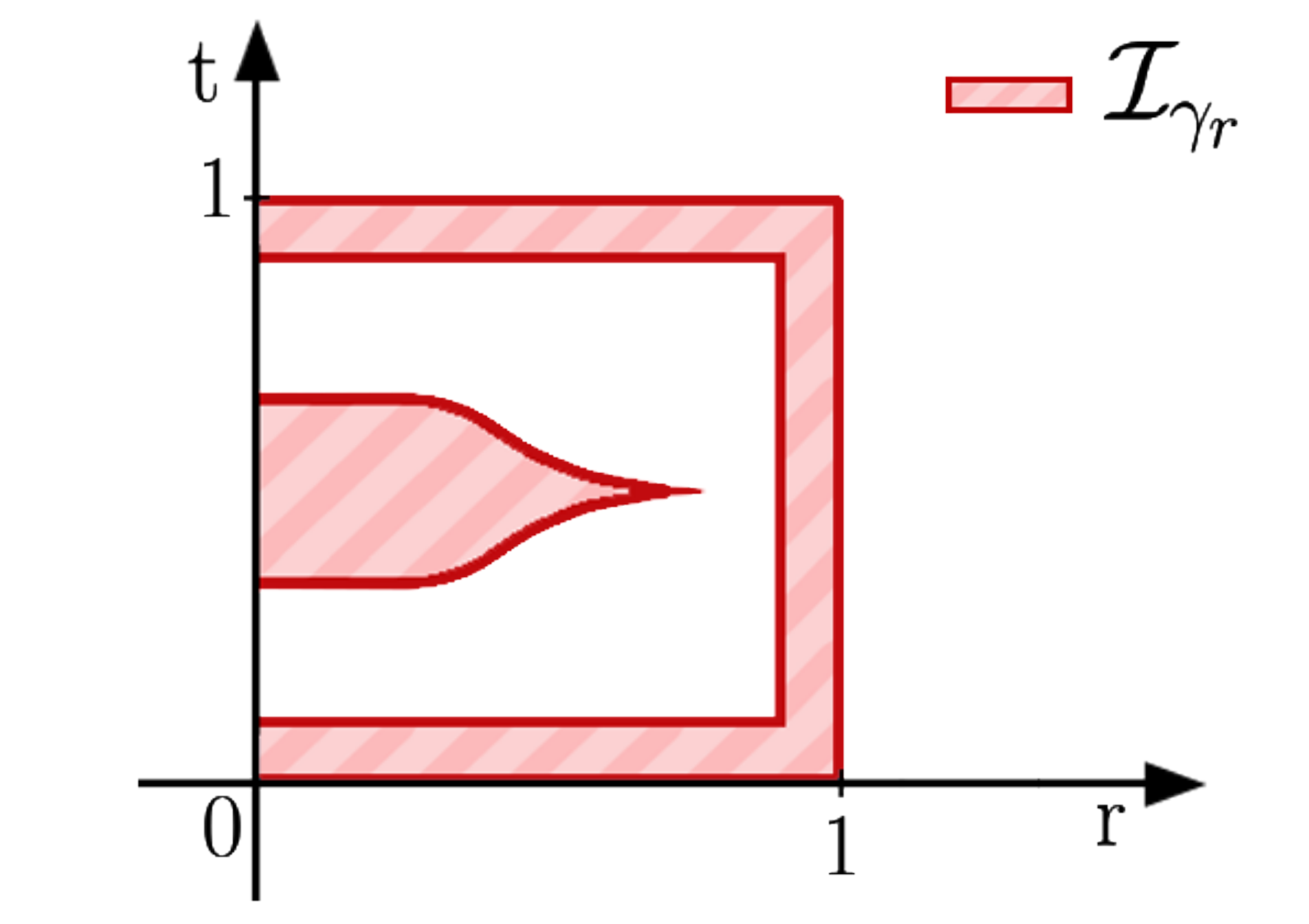}
\caption{Example of inflection points for a tame family similar to the one in the proof of Theorem \ref{thm:fill}.}
%{\newline\newline\bf
%\color{red} Some please change this figure, the blue boundary is also in $\mathcal{I}_{\gamma_r}$. Also, make the fonts bigger. Also, add some variety in the form of horizontal segments.}
\end{figure}
% ------------------------------

In light of Proposition \ref{prop:solidTameComb}, we now focus on deforming any $6\pi$--Engel comb into a tame Engel comb.

\subsection{Reducing to tame Engel combs} \label{ssec:clover}

In this subsection, we first describe two families of curves: the kink and the four--leaf clover. These are then used in the proof of Theorem \ref{thm:fill} to deform a given $K$--Engel comb in the region $t \in [\rho,2\rho]$.

\subsubsection{The kink curve}

By definition a $K$--Engel comb $\gamma_r$ is given in $[0,1] \times [\rho,2\rho]$ by $\Rot(K(t-\rho)/\rho)e_1$, i.e. a rotation along the horizontal equator which does not depend on $p \in \D^3$. For instance, $K\geq6\pi$ results in the curves $\gamma_r$ turning more than $3$ times around the equator at constant speed as $t$ goes from $\rho$ to $2\rho$. The kink curves serve to interpolate, relative to the boundary, between a segment encircling once the horizontal equator and a short convex segment strictly contained in the upper hemisphere, see Figure \ref{fig:kink}.

For each $\theta \in [0, \pi/2]$, consider the plane given by the equation $\{\sin(\theta)(x-1) + \cos(\theta)z = 0\}$. For $\theta = 0$ this describes the plane $\{z = 0\}$ and for $\theta=\pi/2$ the vertical plane $\{x=1\}$. Considering $\theta\in[0,\pi/2)$, the intersection of these planes with the 2--sphere $\S^2$ yields the following parametrised curves
$$\beta_\theta(t) = (\sin^2(\theta)+\cos^2(\theta)\cos(t), \cos(\theta)\sin(t), \sin(\theta)\cos(\theta)(1-\cos(t))), \quad t \in [0,2\pi]. $$
The curve $\beta_0$ parametrises the equator with constant angular speed, and $\beta_{\pi/2}$ is a constant map with image the point $(1,0,0)$. The remaining curves for $\theta\in(0,\pi/2)$ are convex since they present rotational symmetry with respect to the normal axis of the corresponding plane. 
Note also that the Frenet frame remains constant at the origin of these curves: $\FF(\beta_\theta)(0) = \Id$ for $\theta \in [0,\pi/2)$. 

\subsubsection{The four-leaf clover} The geometric reason for us to introduce the four--leaf clover is that it allows to arbitrarily decrease the value of the angular function in $t=2\rho$ at the expense of deforming to convex curves in $t \in [\rho,2\rho]$. Note then that once the angular function is small enough at the point $t=2\rho$, the formal Engel structure $(\D^3 \times [2\rho,1], \SD)$ will be homotopic to a solid Engel shell.

Though the four--leaf clover is a curve on the 2--sphere, it is simpler to describe it in an affine chart.

\begin{lemma}\label{lem:affine}
The affine chart $\pi:\{(x,y,z)\in \NS^2: x>0\}\longrightarrow\R^2$, $\pi(x,y,z)=(y/x,z/x)$ maps geodesics to geodesics, and convex curves to convex curves.
\end{lemma}
\begin{proof}
The map $\pi$ is readily seen to preserve geodesics from the correspondence between geodesics and planes passing through the origin. Convex curves are also preserved because convexity can be defined in terms of the order of contact with the corresponding geodesics.
\end{proof}

The parametrized plane curve $f(t)=(\cos(t)\sin(2t),\sin(t)\sin(2t))$, $t \in [0,2\pi]$, which we call the four--leaf clover, is convex and by Lemma \ref{lem:affine}, so is the curve $(\pi^{-1}\circ f)\subseteq\S^2$. Figure \ref{fig:clover} depicts the clover.

Let us reparametrize the resulting curve to $\kappa(t) = \pi^{-1} \circ f(2\pi t)$, and also reparametrize $\beta(t) = \beta_\theta(6\pi t)$, for a arbitrary but fixed $\theta \in (0, \pi/2)$. Then the curves $\kappa$ and $\beta$ are smoothly homotopic as curves and their Frenet maps agree at the point $t=0$, hence Theorem \ref{thm:Little} implies the following Lemma:

\begin{lemma} \label{lem:kink2clover}
The curves $\beta(t)$ and $\kappa(t)$ are homotopic through a smooth family $\tau_s(t)$ of convex curves, $s \in [0,1]$, with Frenet frames $\mathfrak{F}(\tau_s)(0) = \Id, \forall s\in [0,1]$.
\end{lemma}
\begin{proof}
Both curves $\beta$ and $\kappa$ lie in the same connected component of the space of convex curves since their images by $\pi$ have Gauss maps with winding number $3$. Theorem \ref{thm:Little} provides a smooth family of convex curves $f_s$ joining $f_0 = \beta$ and $f_1 = \kappa$. The family $\tau_s(t)=[\FF(f_s)(0)]^{-1}f_s(t)$ satisfies all the required conditions.
\end{proof}

\begin{figure}
\begin{minipage}[hl]{0.45\textwidth}
\centering
\includegraphics[scale=0.35]{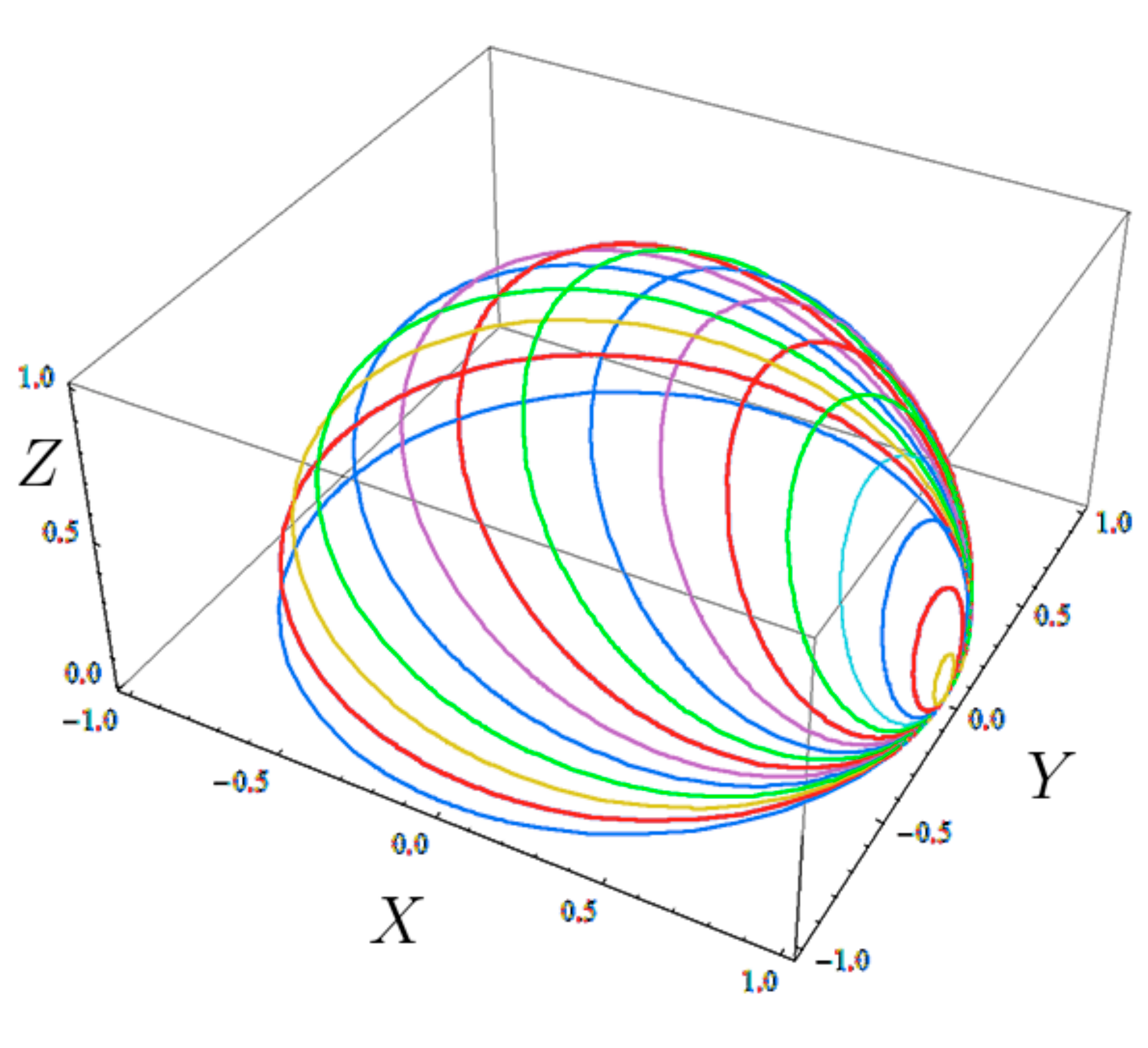}
\caption{The curves $\beta_\theta$ for different values of the parameter $\theta$.}
\label{fig:kink}
\end{minipage}
\begin{minipage}[hr]{0.45\textwidth}
\centering
\includegraphics[scale=0.45]{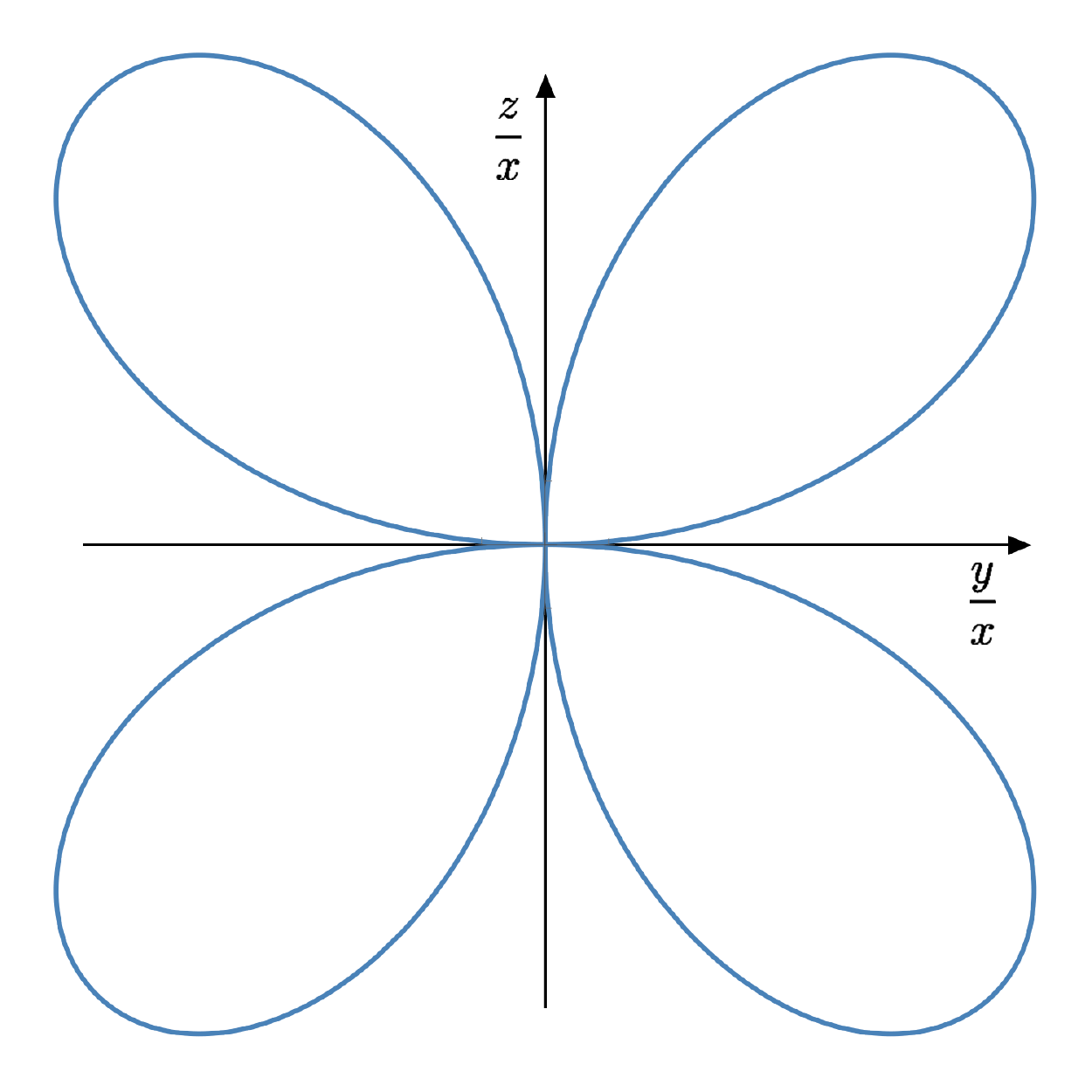}
\caption{The curve $\pi \circ \kappa$.}
\label{fig:clover}
\end{minipage}
\end{figure}

\subsection{The proof of Theorem \ref{thm:fill}}

The argument uses three technical lemmas whose statements and proofs are postponed to the following subsection; their geometric content is however intuitive. Lemmas \ref{lem:glueing1} and \ref{lem:glueing2} state that a family of curves tangent to the equator at a given point can be deformed to be $C^\infty$--tangent to the equator at the point while those that were convex remain so away from it. Lemma \ref{lem:stretch} states that a segment that is $C^\infty$--tangent to the equator on one of its ends can be deformed so that it first parametrizes the original segment and then an arbitrarily long piece of equator. 

Let us now start the argument for Theorem \ref{thm:fill}. Consider $c:\D^3\times[0,1]\longrightarrow\R$ a $6\pi$--radial function. The curves corresponding to its associated $6\pi$--Engel comb $\gamma_r^c$ are tangent to the horizontal equator of the 2--sphere. In case that $c(p,1)-c(p,0)$ is positive on $\D^3$ we can apply Lemma \ref{lem:positive} and we obtain a solid Engel shell. Otherwise the curves $\gamma^c_r$ have points where $(\gamma^c_r)' = 0$.

In short, the argument goes as follows. The only a priori information we have on $\gamma_r^c$ is the existence of a region $t\in[\rho,2\rho]$ in which the curves wind around the horizontal equator three times. The deformation provided by the kink curves $\beta_\theta$ modifies these three laps around the equator into a curve with three kinks. The curve with three kinks can be homotoped to the four--leaf clover curve, which now can be used to arbitrarily decrease the value of $c(p,2\rho)$. The decreasing process consists of clockwise pulling the two left--most leaves of the four--leaf clover around the equator as many times as needed. This process is illustrated in Figure \ref{fig:movie}.

%{\bf\color{red}Figure \ref{fig:movie} should be changed such that the three kinks are created at the same time!}

This geometric explanation is now detailed with the corresponding analysis. First the curve $\beta(t) = \gamma_r^c(\rho(1+t)) = \Rot(6\pi t)e_1$, $t\in[0,1]$, is deformed to the four--leaf clover; this is achieved by applying Lemma \ref{lem:kink2clover} to obtain the family $\tau_s:\NS^1:=[0,1]/\{0 \simeq 1 \}\longrightarrow\NS^2$, $s\in[0,1]$, which we understand as maps with domain the interval $[0,1]$. The family $\tau_s$ can be modified at its ends to glue smoothly with a curve tangent to the equator; this is done by applying Lemma \ref{lem:glueing1} to $\tau_s$ at times $t\in\{0,1\}$, which yields a $[0,1]$--family of curves $f_s$ satisfying that: 
\begin{wideitemize}
\item[-] $f_0(t)=\beta(t)=\Rot(6\pi t)e_1$, for $t\in[0,1]$.
\item[-] There exists a small $\varepsilon>0$ such that $f_s(t)=\tau_s(t)$, for $t\in [\varepsilon, 1-\varepsilon]$ and $s\in[0,1]$.
\item[-] For $s\in(0,1]$, $f_s(t)$ is convex for $t\in(0,1)$ and it has an $\infty$--order of contact with the equator $\{z=0 \}$ at the endpoints $f_s(0)$ and $f_s(1)$.
\item[-] The Frenet frame in the midpoint of the four--leaf clover is
$$\mathfrak{F}(f_1(1/2))= \left(
\begin{array}{ccc}
1&0&0\\
0&-1&0\\
0&0&1
\end{array} \right).$$
\end{wideitemize}
This is the deformation in the region $s\in[0,1]$, we now define a deformation for $s\in[1,2]$.

Note that $t=1/2$ is the time in which the four-leaf clover $\tau_1$ has turned and is pointing in the opposite direction. In order to clockwise pull the two left--most leaves of the four--leaf clover, we first need to flatten the point $t=1/2$ so that it has a tangency of $\infty$--order with the equator: this is done by applying Lemma \ref{lem:glueing2} to the curve $f_1$ at $t=1/2$. This provides a family of curves $f_s: [0,1] \longrightarrow \NS^2$, $s\in [1,2]$ such that:
\begin{wideitemize}
\item[-] There exists a small $\varepsilon>0$ such that $f_s(t)=f_1(t)$, for $t\not \in [1/2-\varepsilon, 1/2+\varepsilon]$.
\item[-] The curves $f_s(t)$, $s\in[1,2)$ are convex if and only if $t\in(0,1)$, and the curve $f_2(t)$ is convex if and only if $t\in(0,1)\setminus\{1/2\}$. The inflection points of $f_s$, $s \in [1,2]$, are $\infty$--order tangencies with the equator $\{z=0\}$.
\item[-] The Frenet frame in the midpoint of these modified four--leaf clovers remains constant:
$$\mathfrak{F}(f_s(1/2))= \left(
\begin{array}{ccc}
1&0&0\\
0&-1&0\\
0&0&1
\end{array} \right).$$
\end{wideitemize}

The $\infty$--order tangency point that we have introduced at $t=1/2$ allows us to stretch the point into an arbitrarily large interval (and hence clockwise pulling the two left--most leaves of the flattened four--leaf clover). This deformation will occur for those values of the parameter $s \in [2,3]$. Consider an arbitrary constant $C<0$ to be chosen later which captures the amount of stretching and clockwise pulling. 

Consider a small $\varepsilon\in\R^+$ and define $\varepsilon_s = (s-2)\varepsilon$. By applying Lemma \ref{lem:stretch} to the flattened four--leaf clover $f_2: [0,1]\longrightarrow\NS^2$ we obtain a family of curves $f_s: [0,1]\longrightarrow\NS^2$, $s\in [2,3]$, satisfying:
\begin{wideitemize}
\item[-] $f_s(t/(1-2\varepsilon_s))=f_2(t)$, for $t\in [0, 1/2-\varepsilon_s]$.
\item[-] $f_s(t)=\Rot(C(s-2))\cdot f_2((t-2\varepsilon_s)/(1-2\varepsilon_s))$, for $t\in [1/2+\varepsilon_s,1]$.
\item[-] The curve $f_s(t)$ negatively winds around the horizontal equator $\{z=0\}$ in the interval $t\in [1/2-\varepsilon_s,1/2+\varepsilon_s]$ with non--vanishing speed.
\end{wideitemize}
The first two conditions just reparametrize curve $f_2(t)$ away from $\Op(\{t=1/2\})$ to a curve $f_3(t)$ such that the beginning remains the same and the end is moved by a clockwise rotation of angle $C$. The third condition is the clockwise--pulling process along the horizontal equator. Figure \ref{fig:movie} describes the family $f_s$, $s\in[0,3]$.
%{\newline\bf\color{red} Figure \ref{fig:movie} needs to be modified.}

\begin{figure}[bp]
\centering
\includegraphics[scale=0.8]{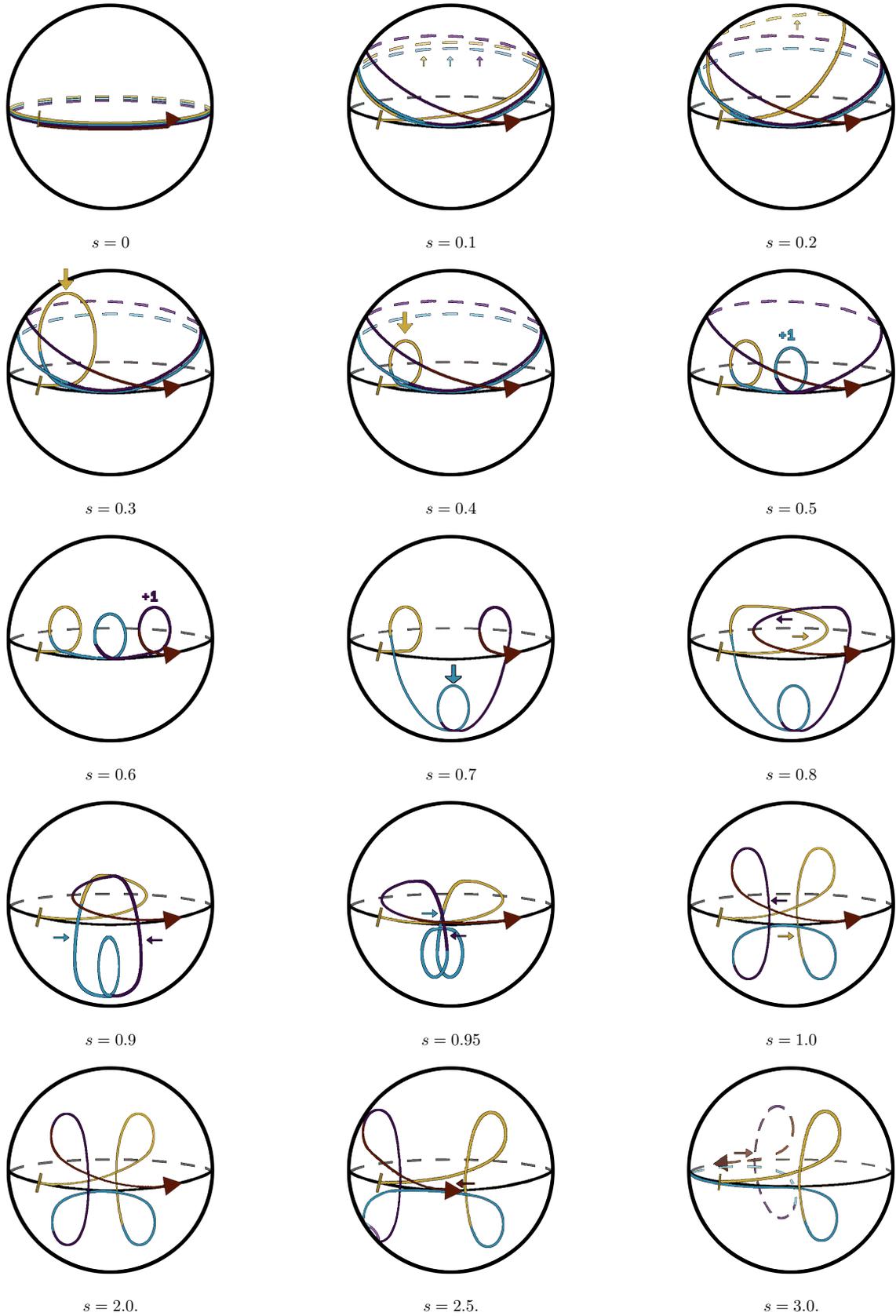}
\caption{The family of curves $f_s$ from the proof of Theorem \ref{thm:fill}.} \label{fig:movie}
\end{figure}

Let us now insert the deformation given by the $[0,3]$--family of curves $f_s$ inside the initial $6\pi$--Engel comb $\gamma_r^c$. The curves in the Engel comb $\gamma_r^c$ have a specified behaviour on the interval $t\in[\rho,2\rho]$ and this is the interval where the deformation provided by $f_s$ is to be inserted. We now provide the analytical details for this.

Consider a small $\delta\in\R^+$ such that the $6\pi$--radial function $c$ is increasing for $|p|\in [1-3\delta,1]$. Define a smooth decreasing cut--off function $\chi:[0,1]\longrightarrow[0,3]$ so that the family $f_{\chi(r)}$ is smooth in the parameter and
$$ \chi(t)=3 \text{ for } t\in[0,1-3\delta], \quad \chi(t)=0 \text{ for } t \in [1-\delta/3,1]. $$

The point now is to replace the initial family of curves of the $6\pi$--Engel comb $\gamma_r^c$ by the family of curves $F_r = f_{\chi(r)}((t-\rho)/\rho)$ in the interval of time $t \in [\rho, 2\rho]$. Observe that they do not glue immediately, since they have differing values at the points $t=\{\rho,2\rho\}$. $|F_r(\rho) - \gamma_r^c(\rho)|$ can be made arbitrarily small according to Lemma \ref{lem:glueing1} and, since $\gamma_r^c$ describes a solid angular shell for $t \in [0,\rho]$, it can be perturbed slightly to allow for the smooth glueing of both families while still describing a solid angular shell for $t \in [0,\rho]$. $F_r(2\rho)$ and $\gamma_r^c(2\rho)$ differ by a rotation of positive angle in the equator and hence we can stretch $\gamma_r^c$ to glue both families at $t=2\rho$. The resulting Engel comb $\Gamma^{\widetilde{c}}_r$ is homotopic through Engel combs to $\gamma_r^c$ and thus provides a deformation of the initial Engel shell.

The Engel shell associated to $\Gamma^{\widetilde{c}}_r$ can be made solid. Indeed, the resulting formal Engel structure is still an Engel structure in the region $[0,2\rho]$ as a consequence of Proposition \ref{prop:solidTameComb}. In the region $t \in [2\rho,1]$ it admits an angle function $\widetilde c(p,t)$ which satisfies $\widetilde c(p,2)\rho)=c(p,2\rho)-C$ (as a consequence of the clockwise--pulling of the two left--most leaves) and $\widetilde c(p,1)=c(p,1)$. The constant $C\in(-\infty,0)$ can then be chosen such that $\widetilde c(p,2\rho) < \widetilde c(p,1)$, and then Lemma \ref{lem:positive} provides a deformation of the Engel shell induced by $\Gamma^{\widetilde{c}}_r$ in the region $\D^3 \times [2\rho,1]$ to a solid Engel shell. This concludes the deformation of the initial $6\pi$--Engel comb into a solid Engel shell and thus proves the statement of Theorem \ref{thm:fill}.\hfill$\Box$

\subsection{Technical lemmas}\label{ssec:Glueing} In the proof of Theorem \ref{thm:fill} we have used two geometric facts regarding deformations of curves in the 2--sphere: 
modification of a horizontal inflection point into an $\infty$--order point of contact with the horizontal equator and the stretching of an $\infty$--order point of contact into an arbitrarily large segment. For completeness, we now include the statements and part of the analytic details of their proofs.

\subsubsection{Two lemmas on achieving $\infty$--order of contact}

The following two lemmas are quite similar in nature, both concerning deformations of a family of curves near a point in order to create $\infty$--order of contact with a certain curve (and at the same time preserving any existing convexity).

\begin{lemma}\label{lem:glueing1}
Consider a smooth family of curves $\gamma_s:[0,1]\longrightarrow\NS^2$, $s \in K$, where $K$ is a compact space. Suppose that the curves $\gamma_s$ are either convex or reparametrizations of an equatorial arc, and the initial Frenet frame is $\mathfrak{F}(\gamma_s)(0)=\Id$ $($hence the initial points $\gamma_s(0)=(1,0,0)$ lie in the horizontal equator$)$.

For any $\varepsilon\in\R^+$ small enough, there is a smooth family of curves $\eta_s: [0,1]\longrightarrow\NS^2$, $s \in K$, satisfying:

\begin{wideitemize}
\item[1.] $\|\eta_s-\gamma_s \|_{C^1} \leq \varepsilon$, $\eta_s|_{[\varepsilon, 1]}=\gamma_s|_{[\varepsilon, 1]}$, and $\FF(\eta_s)(0) = \Rot(-\varepsilon)$.
\item[2a.] If the curve $\gamma_s$ is convex, the curve $\eta_s$ is convex for $t\in(0,1]$ and $\eta_s(0)$ is an $\infty$--order tangency with the horizontal equator $\{z=0\}$.
\item[2b.] If the curve $\gamma_s$ is  a reparametrization of an equatorial arc, so is the curve $\eta_s$.
\end{wideitemize}
\end{lemma}

\begin{proof}
Here we use Lemma \ref{lem:affine} to translate this into a problem of real--valued functions; the affine chart for the 2--sphere is $\pi:H^2\longrightarrow\R^2$. Consider $\delta\in(0,\varepsilon)$ such that $\gamma_s|_{[0,\delta]} \subseteq H^2$ and $(\pi \circ \gamma_s)|_{[0,\delta]}$ are graphical over the horizontal line $\pi(\{z=0\})\subseteq\R^2$. The image of the family of curves $\gamma_s$ can be expressed as a family of plane curves $(t,f_s(t))\subseteq\R^2$, with $f_s:[0,\delta]\longrightarrow\R^+$ a family of smooth functions. It now suffices to appropiately reparametrize the $y$--coordinate $f_s(t)$.

Construct an increasing cut--off function $\chi_1:[-\varepsilon,\delta]\longrightarrow[0,\delta]$ satisfying:
$$\chi^{(k)}(-\varepsilon)=0\mbox{ for }k\in\mathbb{N},\quad\chi''|_{[-\varepsilon,\delta/2)}>0,\mbox{  and  }\chi(t)|_{[\delta/2, \delta]}=t.$$

Since the composition of increasing convex functions is also convex, the family of curves
$$\eta_s: [0,\delta]\longrightarrow H^2,\quad \eta_s(t) = \pi^{-1} \circ (t, f_s\circ \chi(t(1+\varepsilon/\delta)-\varepsilon))$$
preserves any existing convexity and it can be glued with the family of curves $\gamma_s|_{[\delta/2,1]}$; we can then reparametrize in the interval $[\delta/2,\delta]$ to obtained the required family of curves.
\end{proof}

In this same vein, we can smoothly flatten a given point in a convex curve to $\infty$--order of contact with respect to an equator (preserving the Frenet frame at that point). The precise statement reads as follows:

\begin{lemma} \label{lem:glueing2}
Consider a smooth convex curve $\gamma_0:[-1,1]\longrightarrow\NS^2$ with $\mathfrak{F}(\gamma_0)(0)=\Id$ in its midpoint.
For any $\varepsilon\in\R^+$ small enough, there exist smooth curves $\gamma_s: [0,1]\longrightarrow\NS^2$, $s\in[0,1]$, such that
\begin{wideitemize}
\item[1.] $\|\gamma_s-\gamma_0\|_{C^1} \leq \varepsilon$, $\gamma_s|_{[-1,-\varepsilon]\cup[\varepsilon,1]}=\gamma|_{[-1,-\varepsilon]\cup[\varepsilon,1]}$, and $\FF(\gamma_s)(0) = \Id$.
\item[2.] For $s\in(0,1]$, the curves $\gamma_s$ are convex at $t\in[-1, 0) \cup (0,1]$ and the points $\gamma_s(0)$ have $\infty$--order of contact with the horizontal equator $\{z=0\}$.
\end{wideitemize}
\end{lemma}
\begin{proof}
Consider the affine chart in Lemma \ref{lem:affine} and describe the image curve $\pi \circ \gamma$ near the midpoint $t=0$ as the graph of a convex function $f:[-\delta, \delta]\longrightarrow\R^{+}$, for a sufficiently small fixed $\delta>0$. There exist constants $c_0,c_1\in\R^+$ such that
$$0<c_0\leq f''(t),\quad 0 \leq \|f'(t)\|\leq c_1\quad\forall t\in[-\delta,\delta].$$

Given a smooth function $g:[-\delta,\delta]\longrightarrow[-\delta,\delta]$, a condition for $f \circ g$ to be convex is the differential inequality
$$F'' = (f''\circ g)(g')^2 + (f'\circ g)g'' > 0.$$
The bounds given by $c_0,c_1$ above imply that it is sufficient that $g$ satisfies the inequality
$$c_0(g')^2 - |c_1g''| > 0.$$
Let us construct a family $g_s$ of such functions. Consider a function $h:[-\delta,\delta]\longrightarrow[0,1]$ such that
\begin{wideitemize}
\item[a.] $h(-t)=h(t)$.
\item[b.] $h^{(k)}(0)=0$ for $k\in\mathbb{N}$, $h|_{[3\delta/4,1]}= 1$, $h'|_{(0,\delta/4)}>0$ and $h'|_{[\delta/4,\delta/2]}=0$.
\item[c.] $\int_0^\delta h(t)dt = \delta$ and $c_0>|c_1\cdot h'_{|[\delta/2,3\delta/4)}|\geq0$. (by a. we also obtain $c_0>|c_1\cdot h'_{|(-3\delta/4,-\delta/2]}|\geq0$)
\end{wideitemize}
See Figure \ref{fig:graph} for a pictorial description.
%{\newline\bf\color{red} Someone please improve this Figure (and erase the $\mu$), thanks !}

We construct the linear interpolation $g_s(t)= \int_0^t [(1-s) +  sh(t)] dt$, with $s\in[0,1]$, and then the family of curves $\pi^{-1} \circ (t, f \circ g_s(t))$ can be glued with the initial curve $\gamma$ in the region $t \in [-\delta, -3\delta/4] \cup [3\delta/4,\delta]$. \end{proof}

\begin{figure}
\centering
\includegraphics[scale=0.4]{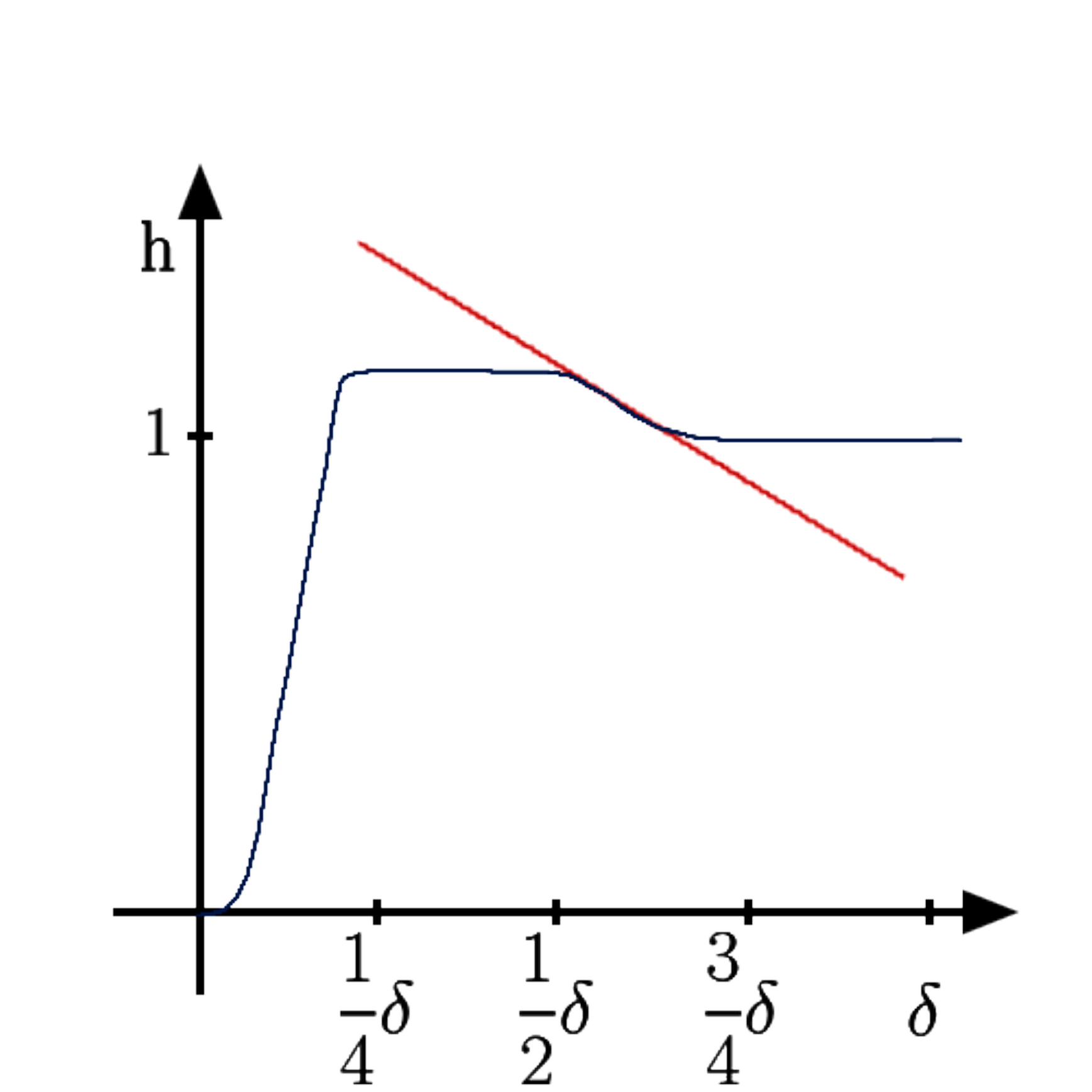}
\caption{Graph of $h$.}
\label{fig:graph}
\end{figure}

\subsubsection{The stretching lemma}
The following lemma concerns the stretching of a flattened point into a segment, the details of the proof are left to the reader.
\begin{lemma} \label{lem:stretch}
Consider a smooth curve $\gamma:[0,1]\longrightarrow\NS^2$ be a curve such that the point $\gamma(1)$ has $\infty$--order of contact with the equator $\{z=0\}$, and the Frenet frame at the endpoint is
$$\FF(\gamma)(1) = \left(\begin{array}{ccc}1&0&0\\0&-1&0\\0&0&1\end{array}\right).$$
Given a smooth function $f:[0,1]\longrightarrow\R$ with $f(0) = 0$ and $f'(0)<0$, the family of curves
\[ \left\{ \begin{array}{l l}
\gamma(t/(1-s/2)), & t \in [0,1-s/2] \\
\Rot(f(s(2t-2+s)))\gamma(1), & t \in [1-s/2,1] 
\end{array} \right. \]
can be reparametrized by a smooth family of smooth curves $\gamma_s: [0,1]\longrightarrow\NS^2$, $s \in [0,1]$.\hfill$\Box$
\end{lemma}

\section{Reducing to the angular model} \label{sec:red}

The proof of Theorem \ref{thm:main} consists of a reduction process and an extension problem. Section \ref{sec:universalHole} defined a particular germ of Engel structure on the boundary of the 4--disk, which Corollary \ref{cor:dominationSymmetric} and Theorem \ref{thm:fill} then extended to an Engel structure on the interior. Hence, in order to conclude Theorem \ref{thm:main}, it is sufficient to homotope a formal Engel structure to a genuine Engel structure except at finitely many 4--disks having such an Engel germ on their boundaries.

The main result of this section is this reduction process, which we state in the following:

\begin{theorem}\label{thm:red}
Let $(\SW_0,\SD_0,\SE_0)$ be a formal Engel structure on a closed 4--fold $M$ and $K\in\R^+$ a constant. Then there exists a homotopy of formal Engel structures $(\SW_t,\SD_t,\SE_t)$, $t\in[0,1]$, and a collection of 4--disks $B_1,\ldots,B_p\subseteq M$ such that:
\begin{wideitemize}
\item[a.] $(\SW_1,\SD_1,\SE_1)$ is a genuine Engel structure in the complement $M\setminus\bigcup_{i=1}^p B_i$.
\item[b.] For each $i\in\{1,\ldots,p\}$, the restriction of the formal Engel structure $(\SW_1,\SD_1,\SE_1)$ to each 4--disk $B_i$ is a $K$--radial shell.
\end{wideitemize}
\end{theorem}

The argument for Theorem \ref{thm:red} uses an adequate triangulation $\Sigma$ of the 4--manifold $M$, and then deforms the formal Engel structure along the 3--skeleton of $\Sigma$ to conform to the two properties in the statement. These two steps are detailed in Subsections \ref{ssec:triang} and \ref{ssec:3sk}, respectively.

\subsection{An adequate triangulation}\label{ssec:triang} 

Consider a 4--manifold $M$ with a formal Engel structure $(\SW,\SD,\SE)$. We construct a triangulation of $M$ adapted to the flowlines of the line field $\SW$. To it, we associate a collection of flowboxes -- closed 4--disks $D^3\times[0,1]\subseteq M$ with coordinates $(x,y,z;t)$ where the line field $\SW$ has the linear description $\partial_t$ -- satisfying a certain nesting property. The specific dimension of the manifold is not important for this argument and hence we will keep it general for later use in the parametric case.

\begin{proposition}\label{prop:triang}
Let $M$ be an $n$--dimensional manifold, $n \geq 2$, endowed with a line field $\SW$. Then there exists a triangulation $\Sigma=\{\sigma\}$ of $M$ and a finite collection $\{S(\sigma)\}_{\sigma \in \Sigma}$ of closed $n$--disks such that
\begin{wideitemize}
\item[a.] Each simplex $\sigma$ is contained in the union $\cup_{\tau\subseteq\sigma}S(\tau)$
\item[b.] The boundary of a simplex $\sigma$ satisfies $\partial\sigma\subset \cup_{\tau\subsetneq\sigma}S(\tau)$.
\item[c.] For each pair of simplices $\sigma,\sigma'$, neither of them containing the other, we have $S(\sigma)\cap S(\sigma')=\emptyset$.

\item[1.] For each simplex $\sigma\in\Sigma$, $\exists\phi(\sigma):S(\sigma)\longrightarrow D^{n-1}\times[0,1]$ such that $\phi(\sigma)_*\SW= \langle \partial_t \rangle$.
\item[2.] For each simplex $\sigma\in\Sigma^{(j)}$, $j<n$, any orbit of the line field $\SW$ in the disk $S(\sigma)$ either avoids the set $\cup_{\tau\subsetneq\sigma}S(\tau)$, or it is entirely contained on it.
\end{wideitemize}
\end{proposition}

Note that the first three properties are of a topological nature, whereas the remaining two requirements belong to a dynamical setting. See Figures \ref{fig:tri2d} and \ref{fig:tri3d} for two and three dimensional examples of the required triangulations.

%{\bf\color{red} Let's improve if possible Figures \ref{fig:tri2d} and \ref{fig:tri3d}.}

\begin{figure}
\begin{minipage}[hl]{0.47\textwidth}
\centering
\includegraphics[scale=0.40]{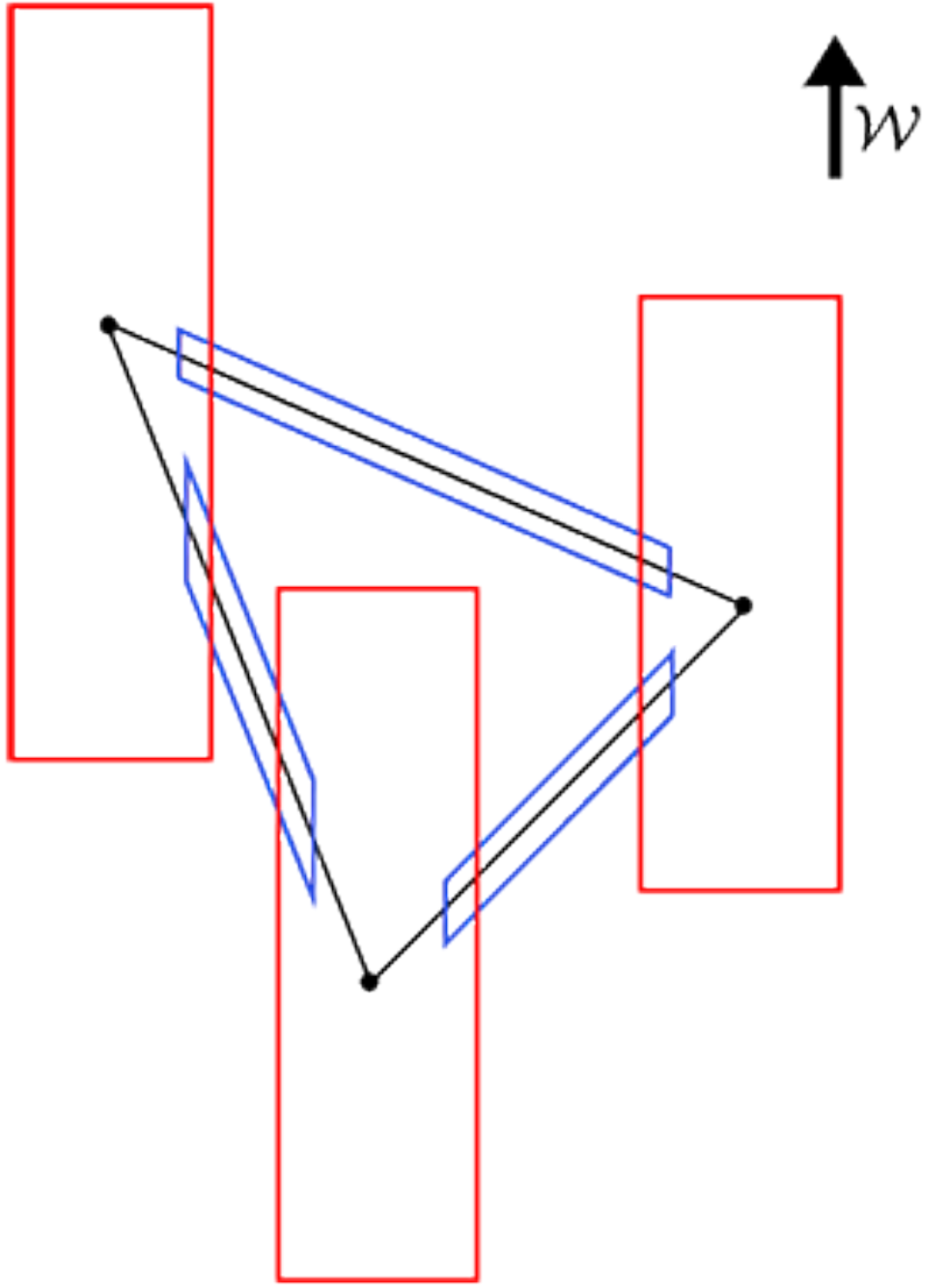}
%{\fontsize{6pt}{6pt}\selectfont 
\caption{Case $n=2$. Red: closed disks for the $0$-simplices. Blue: closed disks for the $1$-simplices.}
\label{fig:tri2d}
\end{minipage}
\begin{minipage}[hr]{0.5\textwidth}
\centering
\includegraphics[scale=0.35]{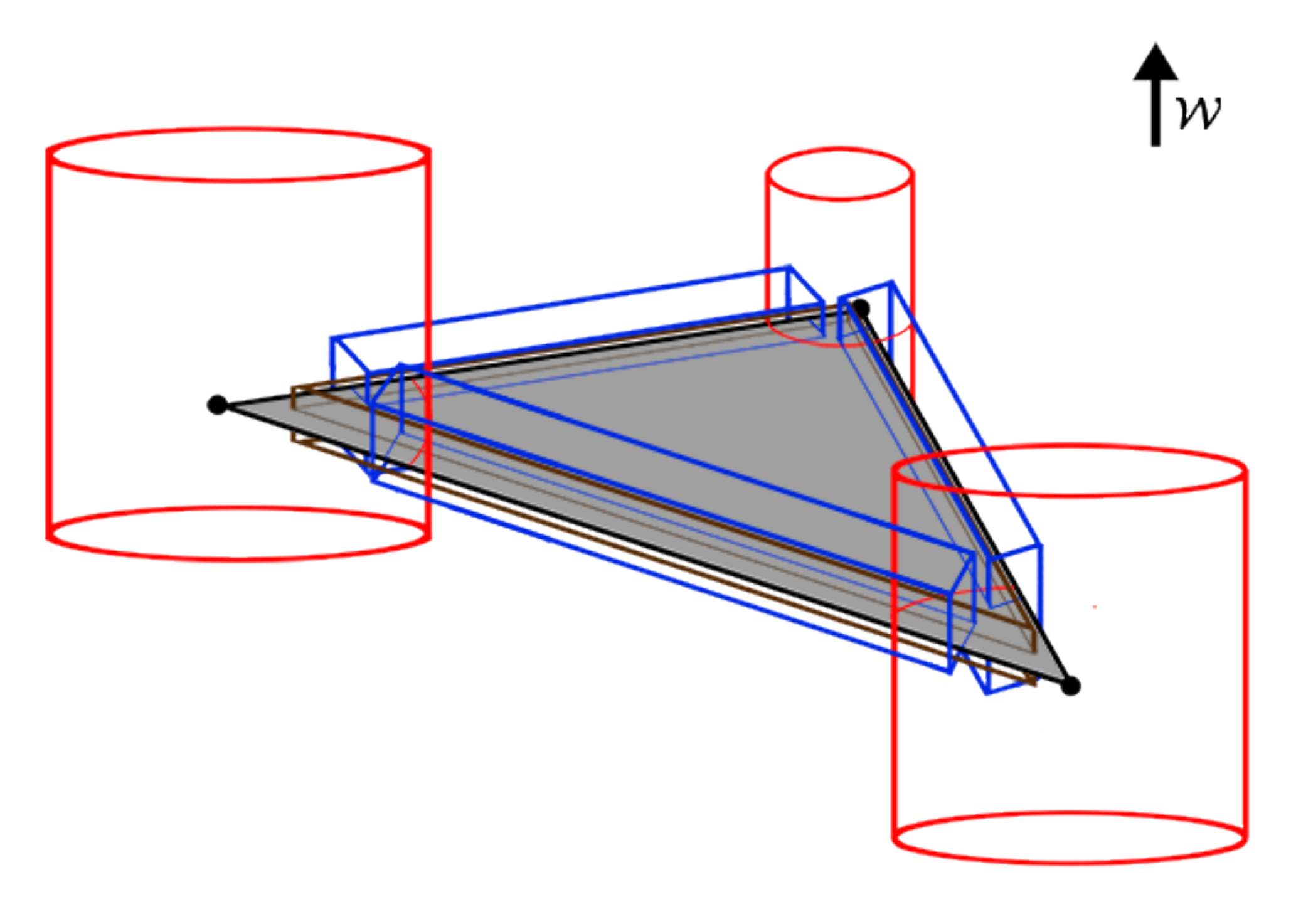}
\caption{ Case $n=3$. Red: closed disks for the $0$-simplices. Blue: closed disks for the $1$-simplices. Brown: closed disks for the $2$-simplices.}
 \label{fig:tri3d}
\end{minipage}
\end{figure}

%\begin{figure}{hl}
%%\begin{minipage}[b]{0.49\textwidth} 
%%\centering
%\includegraphics[scale=0.3]{Triangulacion_2D.pdf} \label{fig:tri2d} 
%\caption{{\fontsize{6pt}{6pt}\selectfont Case $n=2$. Red: closed disks for the $0$-simplices. Blue: closed disks for the $1$-simplices.}}
%\end{figure}
%
%%\begin{minipage}[b]{0.49\textwidth} 
%%\centering
%\begin{figure}
%\includegraphics[scale=0.3]{Triangulacion_3D.pdf} \label{fig:tri3d} 
%\caption{{\fontsize{6pt}{6pt}\selectfont  Case $n=3$. Red: closed disks for the $0$-simplices. Blue: closed disks for the $1$-simplices. Brown: closed disks for the $2$-simplices.}}
%%\end{minipage}
%\end{figure}

\begin{proof}
Fix a Riemannian metric $g$ on the $n$--manifold $M$ and consider a cover of $M$ by open disks such that each disk is a flowbox for the line field $\SW$; we then trivialize $\SW$ in each flowbox by a unitary vector field which we still denote $\SW$. Apply Thurston's Jiggling Lemma \cite[Section 5]{Th} (see also \cite{Thom}) to the 1--distribution $\SW$ in order to find a triangulation $\Sigma$ adapted to the cover such that the line field $\SW$ is transverse to each simplex, i.e.~the angle between the line field and a simplex is strictly positive. 

For each $j$--simplex $\sigma\in\Sigma^{(j)}$, $j<n$, we fix a triple of positive real numbers $(r_0,r_1,r_2)\in\R^+\times\R^+\times\R^+$ on which the set $S(\sigma) \subseteq M$ will depend. Consider a $j$--dimensional disk $\wt\sigma\subseteq\sigma$ such that the distance $r_0< d_g(\partial\wt\sigma,\partial\sigma)\leq 2r_0$. Use the time--$r_1$ exponential map on an orthonormal basis of $(T\sigma\oplus\SW)^\perp\subseteq TM$ and the time--$r_2$ flow of $\SW$ to construct the set:
$$S(\sigma)\cong\wt\sigma\times D^{n-1-j}(r_1)\times[-r_2,r_2]$$
The region of the boundary $\partial(\tilde\sigma\times D^{n-1-j}(r_1))\times[-r_2,r_2]$ will be called the lateral boundary of $S(\sigma)$.
Let us prove that suitable choices of $(r_0,r_1,r_2)$ create a collection $S(\sigma)$ satisfying all the properties required in the statement for $j<n$. The sets $S(\sigma)$ can be chosen to additionally satisfy the following property:
\begin{wideitemize}
\item[3a.] For any $j$--simplex $\sigma$, $j<n$, and any $\tau \subsetneq \sigma$, $\sigma$ intersects the boundary of the set $S(\tau)$ in its lateral region.
\item[3b.] $S(\sigma)$ also intersects the boundary of $S(\tau)$ in its lateral region.
\end{wideitemize}
Which can be readily seen to imply Property (2). We proceed by induction in the dimension of the simplices.

For $j=0$, the first radius $r_0$ is not defined; but observe that the five properties in the statement are satisfied by choosing $r_1,r_2 > 0$ small enough. Further, Property (3) can be satisfied by choosing $r_2 \to 0$ and $r_1/r_2 \to 0$. Indeed, if for each sequence of pairs $(r_1,r_2)$ satisfying $r_2 \to 0$ and $r_1/r_2 \to 0$, Property (3a) does not hold, then the angle between some simplex $\tau$ containing the point $\sigma$ and $\SW$ would be zero, and this is impossible.

Let us explain the inductive step: we suppose that the six properties hold for the $k$--simplices, $k= 0,\ldots,j-1$, and we consider a $j$--dimensional simplex $\sigma$. Choose the first two radii $(r_0,r_1)$ small enough such that
$$\partial(\tilde\sigma \times  D^{n-1-j}(r_1))\subseteq\cup_{\tau\subsetneq\sigma}S(\tau),$$
and shrink $(r_1, r_2)$ to guarantee Property (c). Property (3a) is achieved by choosing the quotient $r_2/r_1$ to be large enough and then Property (3b) is guaranteed if $r_2$ is chosen small enough.

It remains to consider the $n$--dimensional simplices $\sigma\in\Sigma^{(n)}$: for each such $\sigma$ we consider the PL--smooth disk $D_+$ constructed as the union of the faces of $\sigma$ where $\SW$ is inward pointing. This yields a flowbox for $\SW$ contained in $\sigma$ by considering the forward flow (for differing times) of a disk contained in $D_+$; this flowbox can be smoothed and assumed to have boundary $C^0$--close to $\partial\sigma$.
\end{proof}

\subsection{Engel energy}\label{ssec:3sk}
The starting point in this subsection is that of Theorem \ref{thm:red}, a formal Engel structure $(\SW,\SD,\SE)$ on a smooth 4--fold $M$. By applying Theorem \ref{thm:hprincipleEvenContact}, we can suppose that the 3--distribution $\SE^3$ is even contact with the line field $\SW^1$ being its kernel. The deformations considered henceforth maintain both $\SE$ and $\SW$.

Let us introduce a measure of the Engelness of a formal Engel structure, which we refer to as the Engel energy; the argument for Theorem \ref{thm:red} is phrased in terms of the creation of such Engel energy.

Consider an auxiliary Riemannian metric $g$ and the Riemannian orthogonal $\SW^{\perp_g}\subseteq\SE$ of the line field $\SW$ inside the 3--distribution $\SE$. Given a point $p$, fix unitary vectors $W\in\SW$, $X\in \SD \cap \SW^{\perp_g}$ and $Y\in\SW^{\perp_g}$ in a neighbourhood $\Op(p)$ such that $\{ W, X, Y\}$ is an oriented unitary local basis of the 3--distribution $\SE$.

\begin{definition}
The Engel Energy of the 2--distribution $\SD$ at the point $p\in M$ is
\[ \SH(\SD)(p)= \langle \SL_ W X, Y \rangle. \]
\end{definition}
The convention on orientations makes this quantity well--defined; this captures analytically the geometric intuition that in order for $(\SW,\SD,\SE)$ to define an Engel structure, the Legendrian vector field $X$ should rotate (towards $Y$) when we flow along the line field $\SW$. One can also verify the following
\begin{lemma}
Let $(M;\SW,\SD,\SE)$ be a formal Engel structure with $(\SE,\SW)$ even--contact. Then
$$\SH(\SD)(p)>0\Longleftrightarrow (\SW,\SD,\SE)\mbox{ is Engel at }p.$$
\end{lemma}

For a closed domain $U\subseteq M^4$, a chart $\phi:U \longrightarrow \D^3\times[0,1]$ is said to be adapted if $\phi_*\SW= \langle \partial_t \rangle$; the charts associated to the triangulation provided by Proposition \ref{prop:triang} are adapted. Then the Engel energy can be described in terms of the local angle functions introduced in Section \ref{sec:universalHole}:

\begin{lemma} \label{lem:EngelEnergyBounds}
Fix the pair $(\SW,\SE)$ and let $\phi:U\longrightarrow\D^3\times[0,1]$ be an adapted chart. Then there exists a strictly positive function $C_{\phi}(p,t):\D^3\times[0,1]\longrightarrow\R^+$ such that
$$ C_{\phi}(p,t) \cdot \SH(\SD)(\phi^{-1}(p,t)) = \partial_tc(\phi_*\SD)(p,t) $$
for any $2$--plane $\SD$ such that $\SW \subset \SD \subset \SE$.\hfill$\Box$
\end{lemma} 
This concludes the discussion on Engel energy, which is used in the forthcoming proof of Theorem \ref{thm:red}.

\subsection{Proof of Theorem \ref{thm:red}} Consider a formal Engel structure $(\SW_0,\SD_0,\SE_0)$ on a closed 4--fold $M$ and $K\in\R^+$ a constant. By applying Theorem \ref{thm:hprincipleEvenContact} we suppose that the 3--distribution $\SE$ is even--contact and the line field $\SW$ is its kernel. Proposition \ref{prop:triang} provides a triangulation $\Sigma = \{\sigma\}$ and a covering of closed disks $\{S(\sigma)\}$ with useful properties.

The first step is to deform the formal Engel structure $(\SW_0,\SD_0,\SE_0)$ to a formal Engel structure which is Engel near the 3--skeleton $\Sigma^{(3)}$ and contains enough Engel energy; this geometrically translates into the Legendrian vector field rotating sufficiently fast. This is achieved by creating Engel energy inductively on the skeleta of the triangulation $\Sigma$.

\underline{Engel Energy in the lower skeleta.} Consider a positive constant $K_0\in\R^+$. Let us construct a deformation $\SD'$ of $\SD$ satisfying $\SH(\SD')|_{S_3}>K_0$, where we denote
$$S_j:=\bigcup_{\sigma\in\Sigma^{(j)}} S(\sigma).$$
This is achieved by induction over the dimension $j$ of the simplices.

Suppose that $\SD$ has already been deformed on $S_{j-1}$ suitably. For each $j$--simplex $\sigma$, we thicken $S(\sigma)$ into a bigger flowbox and we consider an adapted chart $\phi(\sigma)$ on this thickening, which identifies it with the 4--disk
$$\D^3_{1+\varepsilon} \times [-\varepsilon, 1+\varepsilon],$$
for some small $\varepsilon\in\R^+$, and identifies $S(\sigma)$ with the 4--subdisk $\D^3 \times [0,1]$.

Consider the image through $\phi(\sigma)$ of the finite union $\cup_{\tau\subsetneq\sigma}S(\tau)$. Property (2) of the triangulation $\Sigma$ implies that this closed set can be described as $A\times [-\varepsilon, 1+\varepsilon]$, for some closed set $A$, if the thickening is small enough. The inductive hypothesis $\SH(\SD)|_{S_{j-1}}>K_0$ translates into the inequality
$$\partial_tc(\phi(\sigma)_*\SD)|_{A\times [-\varepsilon, 1+\varepsilon]}>K_0\cdot C_{\phi(\sigma)}|_{A\times [-\varepsilon, 1+\varepsilon]}.$$

Consider a function $h:\D^3_{1+\varepsilon} \times [-\varepsilon, 1+\varepsilon]\longrightarrow\R$ such that
$$h|_{A\times [-\varepsilon, 1+\varepsilon]}=\partial_tc(\phi(\sigma)_*\SD)|_{A\times [-\varepsilon, 1+\varepsilon]}, \quad 
\mbox{ and } h >K_0\cdot C_{\phi(\sigma)}.$$

This function $h$ is the derivative of an angular function for an Engel shell with Engel energy greater than $K_0$, and it agrees with the function $\partial_tc(\phi(\sigma)_*\SD)$ on $S_{j-1}$, where the Engel energy of the 2--distribution $\SD$ is already greater than $K_0$. 

The linear interpolation serves now as the required deformation of $\SD$. In detail, consider a cut--off function $\beta: \D^3_{1+\varepsilon} \times [-\varepsilon,1+\varepsilon] \to [0,1]$ such that
$$\beta|_{\D^3 \times [0,1]}\equiv1,\quad \beta|_{\Op(\partial(\D^3_{1+\varepsilon}\times[-\varepsilon,1+\varepsilon]))}\equiv0,$$
and the angle function $d:\D^3_{1+\varepsilon} \times [-\varepsilon,1+\varepsilon] \to \R$ defined as the linear interpolation
$$d(p,t) = (1-\beta(p,t))c(p,t) + \beta(p,t)\left(c(p,0) + \int_0^t h(p,t) dt\right).$$

Then the two angle functions $c$ and $d$ are isotopic relative to the boundary, and hence $d$ induces a deformation $\SD'$ of the 2--distribution $\SD$ through structures contained in $\SE$ and transverse to $\SW$. By applying this deformation to each $j$--simplex $\sigma\in\Sigma^{(j)}$ and the inductive character of the argument, we obtain a deformation $\SD'$ such that $\SH(\SD')|_{S_3}>K_0$.\hfill$\Box$

This provides a deformation satisfying Property (a) in the statement of Theorem \ref{thm:red}. The second step in the proof of Theorem \ref{thm:red} is thus to translate the Engel energy in the neighbourhood $S_3$ of the 3--skeleton into a $K$--radial shell model for the $4$--cells; note that the constant $K\in\R^+$ is given, whereas the constant $K_0\in\R^+$ in the previous argument can be chosen arbitrarily.

\underline{Engel Energy in the 4-cells.} Consider a $4$--simplex $\sigma\in\Sigma^{(4)}$, a constant $K_0\in\R^+$, and a $2$--plane $\SD$ with $\SW \subset \SD \subset \SE$ such that $\SH(\SD)|_{S_3}>K_0$. Such an $\SD$ exists by the previous inductive argument in the neighbourhood $S_3$.

Property (b) of the triangulation $\Sigma$ ensures that $\partial\sigma\subseteq\cup_{\tau\subsetneq\sigma}S(\tau)$, which implies
$$\partial_tc(\phi(\sigma)_*\SD)\mbox{ }|_{\partial \D^3} > K_0 \cdot C_{\phi(\sigma)}.$$
Choose the constant $K_0\in\R^+$ such that $K_0\cdot\min C_{\phi(\sigma)}>K$: the number of $4$--cells is finite, and thus such a constant $K_0$ exists because the function $C_{\phi(\sigma)}$ is strictly positive. This implies the inequality $c(\phi(\sigma)_*\SD)|_{\partial \D^3} > K$ for the angle function and we can then apply Corollary \ref{cor:dominationSymmetric} to obtain a deformation into a $K$--radial shell. This concludes the proof of Theorem \ref{thm:red}.\hfill$\Box$

\section{Proof of Theorem \ref{thm:main} and its corollaries} \label{sec:proof}

In this section we first detail the proofs of Theorems \ref{thm:main} and \ref{cor:folia}, and then deduce Corollary \ref{cor:cobordism}.

\subsection{Proof of Theorem \ref{thm:main} and Theorem \ref{cor:folia}}\label{ssec:proof}

The $\pi_0$--statement of Theorem \ref{thm:main}, that is, {\em every formal Engel structure can be deformed through formal Engel structures to an Engel structure}, is a consequence of the reduction result Theorem \ref{thm:red} and the extension result Theorem \ref{thm:fill}. Let us introduce the appropriate language for the parametric versions of these results.

Consider a $\S^k$--family of formal foliated Engel structures $(\SW_x,\SD_x,\SE_x)$, $x \in \NS^k$, in a smooth foliated manifold $(M^{m+4}, \SF^4)$. The Cartesian product manifold $W = M \times \NS^k$ is endowed with the product foliation $\SF_W=\coprod_{x \in \NS^k} \SF \times \{x\}$ and then the family $\{(\SW_x,\SD_x,\SE_x)\}_{x \in \NS^k}$ can be understood as a formal foliated Engel structure $(\SW,\SD,\SE)$ in the foliated manifold $(W^{4+m+k}, \SF_W^4)$. Homotoping this formal Engel flag to a genuine Engel flag amounts to deforming the original family of formal foliated Engel structures to a family of genuine foliated Engel structures. 

In consequence, the $\pi_0$--surjectivity of Theorem \ref{cor:folia} applied to the formal foliated Engel structure $(W^{m+4+k},\SF_W^4,\SW,\SD,\SE)$ implies the higher $\pi_k$--surjectivity for the formal foliated Engel structure $(M^{m+4},\SF^4,\SW,\SD,\SE)$. Note that the statement of Theorem \ref{cor:folia} in the case $m=0$ implies Theorem \ref{thm:main}, and thus it suffices to discuss the proof of Theorem \ref{cor:folia}.

The two central ingredients in the proof for the $\pi_0$--surjectivity in Theorem \ref{cor:folia} are Theorem \ref{thm:red} and Theorem \ref{thm:fill} (in order of application). Let us discuss their parametric analogues; the definitions of Engel, angular and $K$--radial shells can be generalized to the foliated case:

\begin{definition}
A formal foliated Engel structure $(\D^3 \times [0,1] \times \D^m,\coprod_{x \in \D^m} \D^3 \times [0,1] \times \{x\};\SW, \SD, \SE)$ is said to be a foliated Engel $($angular or $K$--radial$)$ shell if:
\begin{wideitemize}
\item[a.] $(\D^3 \times [0,1] \times \{x\},\SW,\SD,\SE)$ if an Engel $($angular or $K$--radial$)$ shell for all $x \in \D^m$,
\item[b.] $(\D^3 \times [0,1] \times \{x\},\SW,\SD,\SE)$ is solid for $x \in \Op(\partial \D^m)$.
\end{wideitemize}
A foliated Engel shell is solid if its formal foliated Engel structure is a foliated Engel structure.
\end{definition}
Note that the parameter space in these foliated definitions is the $m$--disk $\D^m$. The parametric generalization of the reduction result Theorem \ref{thm:red} can be stated as follows:

\begin{theorem} \label{prop:redParametric}
Let $(W^{4+m}, \SF^4;\SW_0,\SD_0,\SE_0)$ be a formal foliated Engel structure and $K\in\R^+$ a constant. Then there exists a homotopy of formal foliated Engel structures $(\SW_t,\SD_t,\SE_t)$, $t\in[0,1]$, and a collection of $(4+m)$--disks $B_1,\ldots,B_p\subseteq M$ such that:

\begin{wideitemize}
\item[1.] $(\SW_1,\SD_1,\SE_1)$ is a foliated Engel structure in the complement of $W\setminus\bigcup_{i=1}^p B_i$.
\item[2.] For each $i\in\{1,\ldots,p\}$, $(B_i,\SF|_{B_i};\SW_1,\SD_1,\SE_1)$ is a foliated $K$--radial shell.
\end{wideitemize}
\end{theorem}
\begin{proof}
Theorem \ref{thm:hprincipleEvenContact} provides a deformation of the formal foliated Engel structure $(\SW_0,\SD_0,\SE_0)$ into a formal foliated Engel structure such that $\SE$ is a leafwise even--contact structure and $\SW$ is its leafwise kernel. Proposition \ref{prop:triang} applied to the pair $(W,\SW)$ provides a triangulation $\Sigma$ and an associated cover by sets $\{S(\sigma)\}_{\sigma \in \Sigma}$ such that the closed neighbourhoods $S(\sigma)$ are of the form $\D^3 \times [0,1] \times \D^m$, and are at the same time flowboxes for the line field $\SW$ and foliated charts for the foliation $\SF$. This can be achieved by requiring in its proof that we first follow the exponential flow in the leaf and then in the ambient manifold.

The proof for the non--parametric case works verbatim by observing that in each closed neighbourhood $S(\sigma)$, the angular functions of the leafwise Engel structures can be described by a smooth function
$$c(p,t,x):\D^3\times[0,1]\times\D^m\longrightarrow\R$$
to which the deformations in the non--parametric Theorem \ref{thm:red} can be applied.
\end{proof}
The foliated generalization of the extension result Theorem \ref{thm:fill} reads as follows:
\begin{theorem} \label{prop:fillParametric}
A foliated $6\pi$--radial shell is homotopic through foliated Engel shells to a solid foliated Engel shell.
\end{theorem}
\begin{proof}
Since all the $6\pi$--radial shells have the same model in the interval $t \in [\rho,2\rho]$, the construction in Theorem \ref{thm:fill} can be applied without introducing additional parameters and we obtain Engel shells with four--leaf clover curves in the interval $t \in [\rho,2\rho]$.
% Let us briefly detail this.
% 
% Let us be precise by reviewing the proof of Theorem \ref{thm:fill}. Denote by $\gamma^p_r$ the comb associated to the $6\pi$--radial shell with domain $\D^3 \times [0,1] \times \{p\}$. Let $\eta > 0$ be some small number such that for $|p| > 1-3\eta$, $\gamma^p_r$ defines a solid Engel shell.
% 
% Write $F_r: [0,1] \to \NS^2$, $r \in [0,1]$, denote the family of curves interpolating from the $6\pi$-angle turn around the equator to the clover that rotates $C$ in the opposite direction. Find a bump function $\beta: [0,1] \to [0,1]$ such that $\beta(t) = 0$, $t \in [1-\eta,1]$, and $\beta(t) = 1$, $t \in [0,1-3\eta]$. Replace $\gamma^p_r(t)$ in $[\rho, 2\rho]$ by $F_{(1-\beta(|p|)) + r\beta(|p|)}((t-\rho)/\rho)$. By setting $C > 0$ large enough, which is possible by compactness of $\D^m$, the claim follows.
\end{proof}
Theorem \ref{prop:redParametric} and Theorem \ref{prop:fillParametric} imply the $\pi_0$--statement of Theorem \ref{cor:folia}, which suffices to prove Theorem \ref{thm:main} and the remaining $\pi_k$--surjectivity in Theorem \ref{cor:folia}.\hfill$\Box$

\subsection{Proof of Corollary \ref{cor:cobordism}}
This cobordism statement requires a proof of the reduction Theorem \ref{thm:red} with a relative character; once a relative reduction can be performed, Theorem \ref{thm:fill} implies the statement. Let us explain the relative reduction.

Consider a collar neighbourhood $\Op(\partial M)\cong\partial M \times [0,1)$ and thicken the filling $M$ to
$$\overline{M}:=M\cup_{\partial M\times\{0\}}\partial M\times[-\varepsilon,0];$$
this allows us to modify the formal Engel structure in $\Op(\partial M\times\{0\})$ as an interior open set of the manifold $\overline{M}$. Triangulate $\partial{M}$ and extend this triangulation to the interior of $M$. Proposition \ref{prop:triang} also holds restricted to triangulations of this form, because the simplices contained in the boundary $\partial{M}$ are already transverse to the triangulation and Thurston's Jiggling Lemma has a relative character. This provides suitable neighbourhoods $S(\sigma) \subseteq \overline{M}$ for each simplex $\sigma$ of the triangulation. Then the rest of the proof of Theorem \ref{thm:red} goes through and provides an Engel structure in a neighbourhood $\Op(M)\subseteq\overline{M}$. By construction, the Engel structure close to $\partial M$ is still an angular model that induces the given contact structure.\hfill$\Box$

\end{document}